\documentclass[11pt]{article}

\usepackage[square,authoryear]{natbib}
\usepackage{marsden_article}
\usepackage{epstopdf}
\usepackage[all]{xy}

\newcommand{\contraction}{\vrule height 0pt depth 0.4pt width 3pt
  \vrule height 7pt depth 0.4pt \kern 3pt}
\newcommand{\Bg}{\mathcal{B}\mathfrak{g}}

\title{The Geometry and Dynamics \\ of Interacting Rigid Bodies and Point
  Vortices}
\author{Joris Vankerschaver$^{a,b}$,
Eva Kanso$^c$,
Jerrold E. Marsden$^a$ \\ \mbox{} \\
{\small $^a$ Control and Dynamical Systems} \\
{\small California Institute of Technology MC 107-81,} 
{\small Pasadena, CA 91125} \\ 
{\small $^b$ Department of Mathematical Physics and Astronomy} \\
{\small Ghent University, Krijgslaan 281, B-9000 Ghent, Belgium} \\
{\small $^c$ Aerospace and Mechanical Engineering} \\
{\small University of Southern California,}
{\small Los Angeles, CA 90089} \\ \mbox{} \\
 {\small E-mail: jv@caltech.edu, kanso@usc.edu, jmarsden@caltech.edu}
}
\date{This version: \today}

\begin{document}
\maketitle

\begin{abstract}
  We derive the equations of motion for a planar rigid body of circular shape moving in a 2D perfect fluid with point vortices using symplectic reduction by
  stages.  After formulating the theory as a mechanical system on a
  configuration space which is the product of a space of embeddings
  and the special Euclidian group in two dimensions, we divide out by
  the particle relabelling symmetry and then by the residual rotational and
  translational symmetry.  The result of the first stage reduction is that the
  system is described by a non-standard magnetic symplectic form encoding the
  effects of the fluid, while at the second stage, a careful analysis
  of the momentum map shows the existence of two equivalent Poisson
  structures for this problem.  For the solid-fluid system, we hence
  recover the ad hoc Poisson structures calculated by Shashikanth,
  Marsden, Burdick and Kelly on the one hand, and Borisov, Mamaev, and Ramodanov on the other hand.  As
  a side result, we obtain a convenient expression for the symplectic
  leaves of the reduced system and we shed further light on the interplay between curvatures and cocycles in the description of the dynamics.
  \end{abstract}

\tableofcontents

\section{Introduction}

In this paper, we use symplectic reduction to derive the equations of
motion for a rigid body moving in a two-dimensional fluid with point
vortices.  Despite the fact that this setup is easily described, it
may come as a surprise that these equations were derived (using
straightforward calculations) only recently by \cite{cylvortices} and
\cite{BoMaRa2003}.  In order to see why this is so, despite the long
history of this kind of problems, and in order to set the stage for
our approach, let us trace some of the history of fluid-rigid body
interaction problems.

\paragraph{Rigid Bodies in Potential Flow.}

The equations of motion for a rigid body of mass $m_b$ and inertia
tensor $\mathbb{I}_b$ were first described by \cite{Ki1877} and are given by 
\begin{equation} \label{kirchhoffeq}
  \dot{\mathbf{L}} = \mathbf{A} \times \mathbf{V} \quad \text{and} \quad  \dot{\mathbf{A}} =
  \mathbf{A} \times \mathbf{\Omega},
\end{equation}
where $\mathbf{V}$ and $\mathbf{\Omega}$ are the linear and angular velocity
of the body, while $\mathbf{L}$ and $\mathbf{A}$ are the linear and angular
momentum, related by $\mathbf{L} = m \mathbf{V}$ and $\mathbf{A} = \mathbb{I}
\mathbf{\Omega}$.  Here, $m = m_b + m_f$ is the total mass of the rigid
body, consisting of the mass $m_b$ and the \emph{virtual mass} $m_f$
induced by the fluid.  Similarly, $\mathbb{I} = \mathbb{I}_b + \mathbb{I}_f$, where
$\mathbb{I}_f$ is the virtual inertia tensor due to the fluid.

The main difference between these equations and the Euler equations
governing the motion of a rigid body in vacuum is the appearance of
the non-zero term $\mathbf{A} \times \mathbf{V}$ on the right-hand side of
the equation for $\mathbf{L}$.  In other words, the center of mass no
longer describes a uniform straight trajectory and is a non-trivial
degree of freedom.  From a geometric point of view, the motion of a
rigid body in a potential flow can therefore be considered as a curve
on the special Euclidian group $\operatorname{SE}(3)$ consisting of all translations
and rotations in $\mathbb{R}^3$, where the former describe the orientation
of the body while the latter encode the location of the center of
mass.  

The kinetic energy for the rigid body in a potential flow is a
quadratic function and hence determines a metric on $\operatorname{SE}(3)$.  The
physical motions of the rigid body are geodesics with respect to this
metric.  By noting that the dynamics is invariant under the action of
$\operatorname{SE}(3)$ on itself, the system can be reduced from the phase space
$T^{\ast}  \operatorname{SE}(3)$ down to one on the dual Lie algebra $\mathfrak{se}(3)^{\ast} $:
in this way, one derives the Kirchhoff equations
(\ref{kirchhoffeq}).  
The situation is similar for a planar rigid body moving in a two-dimensional flow:  the motion is then a geodesic flow on $\operatorname{SE}(2)$, the Euclidian group of the plane, and reduces to a dynamical system on $\mathfrak {se}(2)^{\ast}$.  The dynamics of the reduced system is still described by the Kirchhoff equations, but additional simplifications occur: since both $\mathbf{A}$ and $\mathbf{\Omega}$ are directed along the $z$-axis, the right-hand side of the equation for $\mathbf{A}$ is zero.   Throughout this paper, we will deal with the case of planar bodies and two-dimensional flows only.

The procedure of reducing a mechanical system on a Lie group $G$ whose
dynamics are invariant under the action of $G$ on itself is
known as \textbf{\emph{Lie-Poisson reduction}}.  It was pointed out by
\cite{Ar66} that both the Euler equations for a rigid body as well as
the Euler equations for a perfect fluid can be derived using this
approach.  For more about the history of these and related reduction
procedures, we refer to \cite{MarsdenRatiu}.  The geometric outlook on
the Kirchhoff equations, developed by \cite{Leonard1997}, turned out
to be crucial in the study stability results for bottom-heavy
underwater vehicles; see also \cite{PaRoWu2008}.

\paragraph{Point Vortices.}

Another development in fluid dynamics to which the name of Kirchhoff
is associated, is the motion of point vortices in an inviscid flow.  A
\textbf{\emph{point vortex}} is defined as a singularity in the vorticity
field of a two-dimensional flow: $\omega = \Gamma \delta(\mathbf{x} -
\mathbf{x}_0(t))$, where the constant $\Gamma$ is referred to as the
\textbf{\emph{strength}} of the point vortex.  By plugging a superposition of
$N$ point vortices into the Euler equations for a perfect fluid, one
obtains the following set of ODEs for the evolution of the vortex
locations $\mathbf{x}_i$, $i = 1, \ldots, N$:
\begin{equation} \label{vortexeq}
\Gamma_k \frac{d \mathbf{x}_k}{dt} = J
\frac{\partial H}{\partial \mathbf{x}_k}, 
\end{equation}
where $H = - W_G(\mathbf{x}_1, \ldots, \mathbf{x}_N)$ and $W_G$ is the
so-called \textbf{\emph{Kirchhoff-Routh function}}, which is derived using
the Green's function for the Laplacian with appropriate boundary
conditions.  We will return to the explicit form for $W_G$ later.  For
point vortices moving in an unbounded domain, $W_G$ was derived by
Routh and Kirchhoff, whereas \cite{Lin43a} studied the case of vortices in a
bounded container with fixed boundaries.

It is useful to recall here how the point vortex system is related to the dynamics of an inviscid fluid.  In order to do so, we recall the deep insight of 
\cite{Ar66}, who recognized that the motion of a perfect
fluid in a container $\mathcal{F}$ is a geodesic on the group
$\mathrm{Diff}_{\mathrm{vol}}(\mathcal{F})$ of volume-preserving diffeomorphisms of $\mathcal{F}$, 
in much the same way as the motion of a rigid body is a geodesic on the rotation group $\operatorname{SO}(3)$.
The group $\mathrm{Diff}_{\mathrm{vol}}(\mathcal{F})$ acts on itself from the right and leaves the fluid kinetic
energy invariant, a result known as \textbf{\emph{particle relabelling
  symmetry}}.  Hence, by Lie-Poisson reduction, the system can be
reduced to one on the dual $\mathfrak{X}^{\ast} _{\mathrm{vol}}(\mathcal{F})$ of the Lie algebra
of $\mathrm{Diff}_{\mathrm{vol}}(\mathcal{F})$, and the resulting equations are precisely
Euler's equations.

Moreover, any group acts on its dual Lie algebra
through the co-adjoint action, and a trajectory of the reduced
mechanical system starting on one particular orbit of that action is constrained to 
remain on that particular orbit.  \cite{mw1983} showed that for a
perfect fluid, the co-adjoint orbits are labelled by vorticity and,
when specified to the co-adjoint orbit corresponding to the vorticity
of $N$ point vortices, Euler's equations become precisely the vortex
equations (\ref{vortexeq}).  

For a detailed overview of the geometric approach to fluid dynamics, we refer to
\cite{ArnoldKhesin}.  The dynamics of point vortices is treated by different authors, most notable \cite{Sa1992}, \cite{Ne01} and \cite{Aref2007}.  

\paragraph{The Rigid Body Interacting with Point Vortices.}

Given the interest in both rigid bodies and point vortices, it is natural to study the
dynamics of a planar rigid body interacting with $N$ point vortices.
Surprisingly, it wasn't until the recent work of \cite{cylvortices} (SMBK)
and \cite{BoMaRa2003} (BMR) that the equations of motion for this
dynamical system were established.    Both groups proceeded through an ad-hoc calculation to derive the equations of motion, and showed that the resulting equations are in fact Hamiltonian.  However, both sets of equations are different at first sight.

The SMBK equations are formally identical to the Kirchhoff equations
(\ref{kirchhoffeq}) together with the point vortex equations
(\ref{vortexeq}), but the definitions of the momenta $\mathbf{L}, \mathbf{A}$
and the Hamiltonian $H$ are modified to include the effect of the rigid
body (through the ambient fluid) on the point vortices, and vice
versa.  The configuration space is $\mathfrak{se}(2)^{\ast}  \times \mathbb{R}^{2N}$
and the Poisson structure is the sum of the Poisson structures on the
individual factors.
From the BMR point of view, the dynamic variables are the velocity
$\mathbf{V}$ and angular velocity $\mathbf{A}$, together with the locations
of the vortices.  The Hamiltonian is simply the sum of the kinetic
energies for both subsystems, and to account for the interaction between
the point vortices and the rigid body, BMR introduce a non-standard
Poisson structure involving the stream functions of the fluid.

Somewhat miraculously, the equations of motion obtained by SMBK and
BMR turn out to be equivalent: \cite{Sh2005} establishes the existence of a Poisson map
taking the canonical SMBK Poisson structure
into the BMR Poisson structure.  
However, a number of questions therefore remain.  Most importantly, it is not
obvious why one mechanical problem should be governed by two Poisson
structures, which are at first sight very different but turn out to be related by a certain Poisson map.  Moreover, it is entirely non-obvious why this system is Hamiltonian in the first place: in the work of SMBK and BMR, the Hamiltonian structure is derived only afterwards by direct inspection.    

\paragraph{Main Contributions and Outline of this Paper.}  
We shed more light on the issues addressed above by uncovering the 
geometric
structures that govern this problem.  
We begin
this paper by giving an overview of rigid body dynamics and aspects of
fluid mechanics in section~\ref{sec:fluidsolid}.  This material is
mostly well-known and serves to set the tone for the rest of the
paper.  In section~\ref{sec:neumann} we then introduce the so-called
Neumann connection, giving the response of the fluid to an
infinitesimal motion of the rigid body.  This connection has been
described before, but a detailed overview of its properties seems to be
lacking.  In particular, we derive an expression for the curvature of
this connection when the fluid space is an arbitrary Riemannian
manifold, generalizing a result of
\cite{MontgomeryThesis}.  

The remainder of the paper is then devoted to using
reduction theory to obtain the equations of motion for the fluid-solid system
in a systematic way.  We will derive the equations of motion
for this system by reformulating it first as a geodesic flow on the
Cartesian product of the group $\operatorname{SE}(2) $ of translations and rotations
in 2D, and a space of embeddings describing the fluid configurations.  Two symmetry
groups act on this space: the particle relabelling symmetry group
$\mathrm{Diff}_{\mathrm{vol}}(\mathcal{F})$, and the group $\operatorname{SE}(2) $ itself.  Dividing out by
the combined action of these symmetries is therefore an example of
\textbf{\emph{reduction by stages}}.

\begin{figure}[h!]
\begin{center}
\includegraphics[scale=0.6,angle=0]{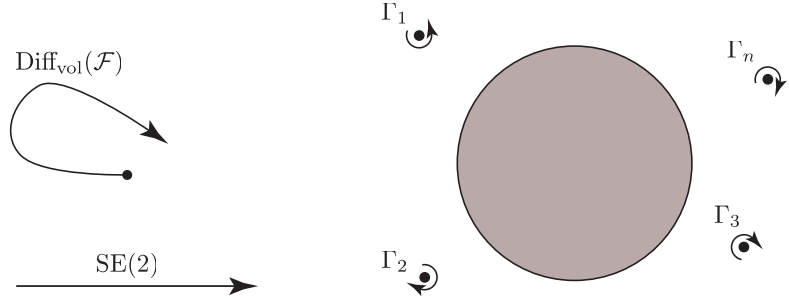}
\end{center}
\caption{\footnotesize
The two symmetry groups in the problem of a cylinder interacting with $N$ point vortices are the particle relabelling group $\mathrm{Diff}_{\mathrm{vol}}(\mathcal{F})$ and the Euclidian group $\operatorname{SE}(2)$.}
\label{fig:diss_energy}
\end{figure}

After symplectic reduction with respect to $\mathrm{Diff}_{\mathrm{vol}}(\mathcal{F})$ in
section~\ref{sec:diff}, we obtain a dynamical system on the product space
$T^{\ast}  \operatorname{SE}(2)  \times \mathbb{R}^{2N}$, where the first factor is the phase
space of the rigid body, while the second factor describes the
locations of the point vortices.  The dynamics is governed by a
\textbf{\emph{magnetic symplectic form}}: it is the sum of the canonical
symplectic forms on both factors, together with a non-canonical interaction term.
From a physical point of view, the interaction term embodies the
interaction between the point vortices and the rigid body through the
ambient fluid.  Mathematically speaking, the map associating to each rigid body
motion the corresponding motion of the fluid can be viewed as a connection, and the
magnetic term is closely related to its curvature.

Note that it is essential here to do \emph{symplectic} reduction with respect to the particle relabeling symmetry rather than Poisson reduction, even though the latter might be conceptually simpler.   Recall that symplectic reduction reduces the system to a co-adjoint orbit; as we shall see below, in the case of fluids interacting with solids these co-adjoint orbits are precisely labelled by the vorticity of the fluid.   Symplectic reduction --- in particular the selection of one particular level set of the momentum map --- therefore amounts to imposing a specific form for the vorticity field.  In our case, this is where we introduce the assumption that the vorticity is concentrated in $N$ point vortices. In the case of Poisson reduction, we would have obtained the equations of motion for a rigid body interacting with an arbitrary vorticity field.

After factoring out the particle relabeling symmetry, the resulting dynamical system is invariant under translations and
rotations in the plane and can then be reduced with respect to the group $\operatorname{SE}(2) $.
This is the subject of section~\ref{sec:redse}.  Physically, this corresponds to rewriting the equations of motion obtained after the first reduction in body coordinates.  However, because of the presence of the   
 magnetic term
in the symplectic form and the fact that $\operatorname{SE}(2) $ acts diagonally on
$T^{\ast}  \operatorname{SE}(2)  \times \mathbb{R}^{2N}$, this is not a straightforward task.  

First, we derive the reduced Poisson structure on 
the reduced space $\mathfrak{se}(2)^{\ast}  \times \mathbb{R}^{2N}$.   Because of the magnetic contributions to the symplectic form, 
the reduced Poisson structure is not just the
sum of the Poisson structures on the individual factors, but includes certain non-canonical contributions as well.   We then show that the momentum map for the $\operatorname{SE}(2)$-symmetry naturally defines a Poisson map taking this Poisson structure into the product
Poisson structure, possibly with a cocycle if the momentum map is not
equivariant (this happens when the total strength of the point vortices is nonzero). 
We do the computations for a general product $T^{\ast}  H
\times P$ first, where $H$ is a Lie group, $P$ is a symplectic
manifold, and the product is equipped with a magnetic symplectic form.
In this way, we generalize the  ``coupling to a Lie group'' scenario
(see \cite{KrMa86}) to the case where magnetic terms are present.

In this way, the results of SMBK and BMR are put on a firm geometric footing: the BMR Poisson structure is the one obtained through reduction and
involves the interaction terms, while the SMBK Poisson structure is
simply the product Poisson structure.  The Poisson map 
induced by the
$\operatorname{SE}(2) $-momentum map described above turns out to be precisely 
Shashikanth's Poisson map.  
As a consequence, we also obtain an explicit  prescription for the
symplectic leaves of this system.

\paragraph{Relation with Other Approaches.}

Our method consists of rederiving the SMBK and BMR equations by reformulating the motion of a rigid body in a fluid as a geodesic problem on the space $Q$.  By imposing the assumption that the vorticity is concentrated in $N$ point vortices, and dividing out the symmetry, we obtain first of all the BMR equations, and secondly (after doing a momentum shift) the SMBK equations.  This procedure is worked out in the body of the paper --- here we would like to reflect on the similarities with other dynamical systems.

Recall that the dynamics of a particle of charge $e$ in a magnetic field $\mathbf{B} = \nabla \times \mathbf{A}$ can be described in two ways.  One is by using canonical variables $(\mathbf{q}, \mathbf{p})$ and the Hamiltonian 
\begin{equation} \label{maghamiltonian}
  H(\mathbf{q}, \mathbf{p}) = \frac{1}{2m} \left\Vert \mathbf{p} - e \mathbf{A}\right\Vert^2,
\end{equation}
while for the other one we use the kinetic energy Hamiltonian $H_{\mathrm{kin}} = \left\Vert \mathbf{p} \right\Vert^2/2m$ but now we modify the symplectic form to be 
\begin{equation} \label{magsymplectic}
\Omega_{\mathbf{B}} = \Omega_{\mathrm{can}} + e\mathbf{B}.  
\end{equation}
A simple calculation shows that both systems ultimately give rise to the familiar Lorentz equations.
From a mathematical point of view, this can be seen by noting that the \textbf{\emph{momentum shift map}}
$\mathcal{S} : \mathbf{p} \mapsto \mathbf{p} - e \mathbf{A}$
maps the dynamics of the former system into that of the latter:
\[
  \mathcal{S}^\ast \Omega_{\mathbf{B}} = \Omega_{\mathrm{can}}
  \quad \text{and} \quad \mathcal{S}^\ast H_{\mathrm{kin}} = H.
\]
In other words, both formulations are related by a symplectic isomorphism, thus making them equivalent.

An over-arching way of looking at the dynamics of a charged particle in a magnetic field is as a geodesic problem (with respect to a certain metric) on a higher-dimensional space: this is part of the famed \textbf{\emph{Kaluza-Klein approach}}.  In this case, spacetime is replaced by the product manifold $\mathbb{R}^4 \times \operatorname{U}(1)$ and charged particles trace out geodesics on this augmented space.  The standard, four-dimensional formulation of the dynamics can then be obtained by dividing out by the internal $\operatorname{U}(1)$-symmetry, resulting in the familiar Lorentz equations on $\mathbb{R}^4$.  More information about these constructions can be found in \cite{Sternberg1977}, \cite{Weinstein1977} and \cite{MarsdenRatiu}.
 
 Our approach to the fluid-structure problem is similar but more involved because of the presence of two different symmetry groups.  However, the underlying philosophy is the same:  by reformulating the motion of a rigid body and the fluid as a geodesic problem on an infinite-dimensional manifold, we follow the philosophy of Kaluza-Klein 
 of trading in the complexities in the equations of motion for an increase in the number of dimensions of the configuration space.  Then, by dividing out the $\mathrm{Diff}_{\mathrm{vol}}$-symmetry, we obtain a reduced dynamical system governed by a magnetic symplectic form (the analogue of (\ref{magsymplectic})), which is mapped after a suitable momentum shift $\mathcal{S}$ into an equivalent dynamical system governed by the canonical symplectic form but with a modified Hamiltonian.   After reducing by the residual $\operatorname{SE}(2)$-symmetry, the former gives rise to the BMR bracket, while the latter is nothing but the SMBK system.
 
\paragraph{A Note on Integrability.}  The case of a circular cylinder is distinguished because of the existence of an additional conservation law, associated with the symmetry that rotates the cylinder around its axis.  When the circular cylinder interacts with one external point vortex, a simple count of dimensions and first integrals suggests that this problem is integrable, a fact first proved by \cite{BoMa03}.  Indeed, for a single vortex the phase space is a symplectic leaf of $\mathfrak{se}^\ast(2) \times \mathbb{R}^2$, which is generically four-dimensional.  On the other hand, for the circular cylinder there are two conservation laws: the total energy and the material symmetry associated with rotations around the axis of the cylinder, which hints at Liouville integrability.   

In the case of an ellipsoidal cylinder with nonzero eccentricity, \cite{BoMaRa2007} gather numerical evidence to show that the interaction with one vortex is chaotic.  This is to be contrasted with the motion of point vortices in an unbounded domain (see \cite{Ne01}), which is integrable for three vortices or less.

\paragraph{Acknowledgements.}  It is a pleasure to thank Richard
Montgomery, Paul Newton, Tudor Ratiu and Banavara Shashikanth, as well
as Frans Cantrijn, Scott Kelly, Bavo Langerock, Jim Radford, and
Clancy Rowley, for useful suggestions and interesting discussions.

J. Vankerschaver is a Postdoctoral Fellow from the Research Foundation
-- Flanders (FWO-Vlaanderen) and a Fulbright Research Scholar, and
wishes to thank both agencies for their support.  Additional financial
support from the Fonds Professor Wuytack is gratefully acknowledged.
E. Kanso's work is partially supported by NSF through the award CMMI
06-44925.  J. E. Marsden is partially supported by NSF Grant
DMS-0505711.

\section{The Fluid-Solid Problem} \label{sec:fluidsolid}

This section is subdivided into two parts.  In the first, and longest,
part we describe the general setup for a planar rigid body interacting
with a 2D flow.  The second part is then devoted to discussing a
number of simplifying assumptions that will make the subsequent
developments clearer.  By separating these assumptions from the main
problem setting, we hope to convince the reader that the method
outlined in this paper does not depend on any specific assumptions on the rigid
body or the fluid, and can be generalized to more complex problems.  At the end of this paper, we discuss how these assumptions can be relaxed.

\subsection{General Geometric Setting}

\paragraph{Kinematics of a Rigid Body.}

Throughout this paper, we consider the motion of a \textbf{\emph{planar rigid
body}} interacting with a \textbf{\emph{2D flow}}.  We introduce an inertial frame
$\mathbf{e}_{1, 2, 3}$, where $\mathbf{e}_{1,2}$ span the plane of motion
and $\mathbf{e}_3$ is perpendicular to it.  The configuration of the
rigid body is then described by a rotation with angle $\beta$ around
$\mathbf{e}_3$ and a vector $\mathbf{x}_0 = x_0 \mathbf{e}_1 + y_0
\mathbf{e}_2$ describing the location of a fixed point of the body, which
we take to be the center of mass.

\begin{figure}[h!]
\begin{center}
\includegraphics[scale=0.6,angle=0]{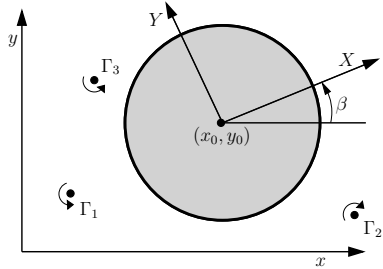}
\end{center}
\caption{\footnotesize
The configuration of the rigid body and the point vortices.}
\label{fig:axes}
\end{figure}

The orientation and position $(\beta; x_0, y_0)$ determine an
element $g$ of the Euclidian group $\operatorname{SE}(2) $ given by
\begin{equation} \label{eucgroup}
  g = \begin{pmatrix}
    R & \mathbf{x}_0 \\
    0 & 1
  \end{pmatrix}, 
  \quad \text{where} \,  
  R = \begin{pmatrix}
    \cos \beta & - \sin \beta \\
    \sin \beta & \cos \beta
  \end{pmatrix}.
\end{equation}

Written in this way, the group composition and inversion in $\operatorname{SE}(2) $
are given by matrix multiplication and inversion.  The Lie algebra of
$\operatorname{SE}(2) $ is denoted by $\mathfrak{se}(2)$ and essentially consists of
infinitesimal rotations and translations.  Its elements are matrices
of the form
\[
  \begin{pmatrix}
    0 & -\Omega & V_x \\
    \Omega & 0 & V_y \\
    0 & 0 & 0
  \end{pmatrix}.
\]
It follows that $\mathfrak{se}(2)$ can be identified with $\mathbb{R}^3$ by
mapping such a matrix to the triple $(\Omega, V_x, V_y)$.  We define a basis 
of $\mathfrak{se}(2)$ given by 
\begin{equation} \label{sebasis}
  \mathbf{e}_\Omega = \begin{pmatrix} 0 & -1 & 0 \\
    1 & 0 & 0 \\
    0 & 0 & 0 
    \end{pmatrix}, \quad
\mathbf{e}_x = \begin{pmatrix} 0 & 0 & 1 \\
    0 & 0 & 0 \\
    0 & 0 & 0 
    \end{pmatrix}, \quad 
\mathbf{e}_y = \begin{pmatrix} 0 & 0 & 0 \\
    0 & 0 & 1 \\
    0 & 0 & 0 
    \end{pmatrix}. 
\end{equation}

For future reference, we also introduce a moving frame fixed to the
rigid body, denoted by $\mathbf{b}_{1, 2, 3}$.  The transformation from
body to inertial frame is given by $\mathbf{x} = R \mathbf{X} +
\mathbf{x}_0$, where $\mathbf{x} = x \mathbf{e}_1 + y \mathbf{e}_2$ and
$\mathbf{X} = X \mathbf{b}_1 + Y \mathbf{b}_2$.

The \textbf{\emph{angular}} and \textbf{\emph{translational velocities}} of the
rigid body relative to the inertial frame are defined as
\begin{equation} \label{spatialvelocity}
\boldsymbol{\omega} = \dot{\beta} \mathbf{e}_3 \quad \text{and} \quad  \mathbf{v} =
\dot{x}_\circ \mathbf{e}_1 + \dot{y}_\circ \mathbf{e}_2,
\end{equation}
where dots denote derivatives with respect to time.  These
quantities can be expressed in the body frame: the body angular
velocity $\boldsymbol{\Omega}$ and the body translational velocity
$\mathbf{V}$ are related to the corresponding inertial quantities by 
\begin{equation} \label{velocities}
  \boldsymbol{\Omega} = \boldsymbol{\omega} \quad \text{and} \quad  
  \mathbf{V} = R^T \mathbf{v}.
\end{equation}
For the case of a planar rigid body, the angular velocity is oriented
along the axis perpendicular to the plane and so determines a scalar
quantity: $\boldsymbol{\omega} = \omega \mathbf{e}_3$ and
$\boldsymbol{\Omega} = \Omega \mathbf{b}_3$.  From a group-theoretic
point of view, if the motion of a rigid body is given by a curve $t
\mapsto g(t)$ in $\operatorname{SE}(2) $, then we may define an element $\xi(t) \in
\mathfrak{se}(2)$ by putting
\[
\xi(t) = \frac{d}{d\epsilon} g(t)^{-1}g(t+\epsilon) \Big|_{\epsilon =
  0},
\]
and it can easily be checked that $\xi$ coincides with the body
angular and translational velocities $(\Omega, \mathbf{V})$.  

The \textbf{\emph{kinetic energy}} of the rigid body is given by the following
expression: 
\begin{equation} \label{bodyenergy}
   T_{\mathrm{body}}(\Omega, \mathbf{V}) = \frac{\mathbb{I}}{2} \Omega^2 + \frac{m}{2} \mathbf{V}^2
\end{equation}
where $\mathbb{I}$ is the moment of inertia of the body and $m$ is its mass.
The kinetic energy defines an inner product on $\mathfrak{se}(2)$, given by 
\begin{equation} \label{kinenergy}
  \left\langle \! \left\langle(\Omega_1, \mathbf{V}_1), (\Omega_2, \mathbf{V}_2)\right\rangle \!
  \right\rangle_{\mathfrak{se}(2)} = 
    \begin{pmatrix} 
      \Omega_1 & \mathbf{V}_1
    \end{pmatrix}
      \mathbb{M}_m
    \begin{pmatrix} 
      \Omega_2 \\ \mathbf{V}_2
    \end{pmatrix}, 
\end{equation}
for $(\Omega_1, \mathbf{V}_1)$ and $(\Omega_2, \mathbf{V}_2)$ in
$\mathfrak{se}(2)$, and where $\mathbb{M}_m$ is given by 
\begin{equation} \label{Mm}
\mathbb{M}_m = 
\begin{pmatrix}
      \mathbb{I} & 0 \\
      0 & m \mathbf{I} 
    \end{pmatrix}.
\end{equation}
Here, $\mathbf{I}$ is the $2$-by-$2$ identity matrix.  By left
extension, this inner product induces a left-invariant metric on the
whole of $\operatorname{SE}(2) $:
\begin{equation} \label{metric:se}
   \left\langle \! \left\langle (g, \dot{g}_1), (g, \dot{g}_2)\right\rangle \!
  \right\rangle_{\operatorname{SE}(2) } :=
     \left\langle \! \left\langle g^{-1}\dot{g}_1, g^{-1}\dot{g}_2\right\rangle \!
  \right\rangle_{\mathfrak{se}(2)},
\end{equation}
for $(g, \dot{g}_1)$ and $(g, \dot{g}_2)$ in $T\operatorname{SE}(2) $.

\paragraph{Incompressible Fluid Dynamics.}  The geometric description
of an incompressible fluid goes back at least to the work of
\cite{Ar66}, who described the motion of an incompressible fluid in a
fixed container $\mathcal{F}$ as a geodesic on the diffeomorphism group of
$\mathcal{F}$.  In the case of a fluid interacting with a rigid body, the
fluid container may change over time, reflecting the fact that the
rigid body moves.

Arnold's formulation can be extended to cover this case by considering
the \textbf{\emph{space of embeddings}} of the reference configuration of the
fluid (denoted by $\mathcal{F}_0$) into $\mathbb{R}^2$.  Recall that an embedding
$\varphi: \mathcal{F}_0 \hookrightarrow \mathbb{R}^2$ maps each reference point $X
\in \mathcal{F}_0$ to its current configuration $x = \varphi(X)$.  In order
to reflect the fact that the fluid is taken to be incompressible, we
require that any fluid embedding $\varphi$ is \textbf{\emph{volume
  preserving}}: if $d \mathbf{x}$ is the Euclidian area element on
$\mathbb{R}^2$ and $\eta_0$ is a fixed area element on $\mathcal{F}_0$, then we require
\begin{equation} \label{volforms}
  \varphi^{\ast}  (d \mathbf{x}) = \eta_0.
\end{equation}
The space of all such volume-preserving embeddings is denoted by
$\operatorname{Emb}(\mathcal{F}_0, \mathbb{R}^2)$.  In the sequel, we will specify additional
boundary conditions on the fluid configurations, stating for example
that the fluid is free to slide along the boundary of the solid, but
this does not make any difference for the current expository
treatment.

A \textbf{\emph{motion}} of a fluid is described by a curve $t \mapsto
\varphi(t)$ in $\operatorname{Emb}(\mathcal{F}_0, \mathbb{R}^2)$.  The \textbf{\emph{material velocity
  field}} is the tangent vector field $(\varphi_t, \dot{\varphi}_t)$
along the curve.  Here, $\dot{\varphi}_t$ is a map from $\mathcal{F}_0$ to
$T\mathbb{R}^2$, whose value at a point $x \in \mathcal{F}_0$ is given by
\[
   \dot{\varphi}_t(x) = \frac{d}{d t} \varphi_t(x) \in T_{\varphi(x)} \mathbb{R}^2.
\]

Note that $\dot{\varphi}_t$ is not a vector field in the traditional
sense; rather, it is a vector field along the map $\varphi_t$.  In
contrast, the \textbf{\emph{spatial velocity field}} $\mathbf{u}_t$, defined as
\begin{equation} \label{eulervelocity}
  \mathbf{u}_t = \dot{\varphi}_t \circ \varphi_t^{-1}
\end{equation}
is a proper vector field, defined on $\varphi_t(\mathcal{F}_0)$.  

The motion of a fluid can be described using the kinetic-energy
Lagrangian:
\begin{equation} \label{fluidenergy}
  T_{\mathrm{fluid}}(\varphi, \dot{\varphi}) = \frac{\rho}{2} \int_{\mathcal{F}_t} \left\Vert\mathbf{u}\right\Vert^2 d
  \mathbf{x}, 
\end{equation}
where $\mathbf{u}$ is the Eulerian velocity field (\ref{eulervelocity})
and the integration domain is the spatial domain of the fluid at time
$t$: $\mathcal{F}_t = \varphi_t(\mathcal{F}_0)$.  Just as in the case of the rigid
body, this kinetic energy induces a metric on the space $\operatorname{Emb}(\mathcal{F}_0,
\mathbb{R}^2)$, given by
\begin{equation} \label{embmetric} \left\langle \! \left\langle \dot{\varphi}_1,
    \dot{\varphi}_2 \right\rangle \!
  \right\rangle_{\operatorname{Emb}} = \rho \int_{\mathcal{F}_0}
  \dot{\varphi}_1 \cdot \dot{\varphi}_2 \, \eta_0.
\end{equation}
By changing variables, this metric can be rewritten in spatial form as
follows:
\[
\left\langle \! \left\langle\dot{\varphi}_1, \dot{\varphi}_2\right\rangle \!
  \right\rangle_{\operatorname{Emb}}
= \rho \int_{\mathcal{F}_t} \mathbf{u}_1 \cdot \mathbf{u}_2 \, d \mathbf{x}, 
\]
where $\mathbf{u}_i$ is the Eulerian velocity associated to
$\dot{\varphi}_i$: $\mathbf{u}_i = \dot{\varphi}_i \circ \varphi_i^{-1}$,
for $i = 1, 2$, and the integration is again over the spatial domain
$\mathcal{F}_t := \varphi_t(\mathcal{F}_0)$ of the fluid.

\paragraph{The Configuration Space of the Fluid-Solid System.}

The motion of a rigid body in an incompressible fluid combines aspects of both
rigid-body and fluid dynamics.  We assume that the body occupies a
circular region $\mathcal{B}_0$ in the reference configuration, and that the
remainder of the domain, denoted by $\mathcal{F}_0$, is taken up by the
fluid.  The configuration space for the fluid-solid system is made up
of pairs $(g, \varphi) \in \operatorname{SE}(2)  \times \operatorname{Emb}(\mathcal{F}_0, \mathbb{R}^2)$
satisfying the following conditions.

\begin{enumerate}
\item The embedding $\varphi$ represents the configuration of the
  fluid.  In particular, $\varphi$ is volume-preserving, \emph{i.e.}
  (\ref{volforms}) is satisfied.  In addition, we assume that
  $\varphi$ approaches the identity at infinity suitably
  fast.\footnote{We will not be concerned with any functional-analytic
    issues concerning these infinite-dimensional manifolds of
    mappings.  Instead, the reader is referred to \cite{EbinMarsden70}
    for more information.}

\item The element $g \in \operatorname{SE}(2) $ describes the configuration of the
  rigid body.

\item The fluid satisfies a ``slip'' boundary condition: the normal
  velocity of the fluid coincides with the normal velocity of the
  solid, while the tangent velocity can be arbitrary, reflecting the
  fact that there is no viscosity in the fluid.  Mathematically, this
  boundary condition is imposed by requiring that, as sets,
  $\varphi(\partial \mathcal{F}_0)$ is equal to $g(\partial \mathcal{B}_0)$, where
  $g \in \operatorname{SE}(2) $ is interpreted as a linear embedding of $\mathcal{B}_0$ into
  $\mathbb{R}^2$.
\end{enumerate}

We denote the space of all such pairs as $Q$; this is a submanifold of
$\operatorname{SE}(2)  \times \operatorname{Emb}(\mathcal{F}_0, \mathbb{R}^2)$.  The kinetic energy of the
fluid-solid system is given by the sum of the rigid-body energy
$T_{\mathrm{body}}$ and the kinetic energy $T_{\mathrm{fluid}}$ of the fluid:
\begin{align}
  T & = T_{\mathrm{fluid}} + T_{\mathrm{body}} \nonumber \\
    & = \frac{\rho}{2} \int_{\mathcal{F}_t} \left\Vert\mathbf{u}\right\Vert^2 d
  \mathbf{x} + \frac{\mathbb{I}}{2} \Omega^2 + \frac{m}{2}
  \mathbf{V}^2. \label{totalenergy} 
\end{align}
Similarly, there exists a metric on $Q$ given by the sum of the
metrics (\ref{metric:se}) and (\ref{embmetric}):
\begin{equation} \label{metric}
  \left\langle \! \left\langle(\dot{\varphi}_1, \dot{g}_1), (\dot{\varphi}_2, \dot{g}_2)\right\rangle \!
  \right\rangle
  = 
 \left\langle \! \left\langle \dot{\varphi}_1, \dot{\varphi}_2 \right\rangle \!
  \right\rangle_{\operatorname{Emb}} +
 \left\langle \! \left\langle \dot{g}_1, \dot{g}_2 \right\rangle \!
  \right\rangle_{\operatorname{SE}(2) }.
\end{equation}

The dynamics of rigid bodies moving in perfect fluids was studied
before by \cite{Ke1998}, \cite{Ra2003}, and \cite{KMlocomotion,
  KaOs08}.  A similar configuration space, but with the $\operatorname{SE}(2) $-factor
replaced by a suitable set of smooth manifolds, was studied by
\cite{freeboundary} for the dynamics of a liquid drop.

\paragraph{Particle Relabelling Symmetry.}

The kinetic energy $T_{\mathrm{fluid}}$ of the fluid is invariant if we
replace $\dot{\varphi}$ by $\dot{\varphi} \circ \phi$, where $\phi$ is
a volume-preserving diffeomorphism from $\mathcal{F}_0$ to itself.  This
represents the \textbf{\emph{particle relabelling symmetry}} referred to in the
introduction. 

Recall that a diffeomorphism $\phi: \mathcal{F}_0 \rightarrow \mathcal{F}_0$ is
volume-preserving if $\phi^{\ast}  \eta_0 = \eta_0$, where $\eta_0$ is the volume element on $\mathcal{F}_0$.  The group of all
volume-preserving diffeomorphisms is denoted by $\mathrm{Diff}_{\mathrm{vol}}$.  This group acts on
the right on $\operatorname{Emb}(\mathcal{F}_0, \mathbb{R}^2)$ by putting $\varphi \cdot \phi =
\varphi \circ \phi$, and hence also on $Q$.  The action of
$\mathrm{Diff}_{\mathrm{vol}}$ on $Q$ makes $Q$ into the total space of a principal
fiber bundle over $\operatorname{SE}(2) $.  In other words, if we define the
projection $\pi: Q \rightarrow \operatorname{SE}(2) $ as being the projection onto the
first factor: $\pi(g, \varphi) = g$, then the fibers of $\pi$
coincide precisely with the orbits of $\mathrm{Diff}_{\mathrm{vol}}$ in $Q$. 

\paragraph{Vorticity and Circulation.}

In classical fluid dynamics, the \textbf{\emph{vorticity}} $\boldsymbol{\mu}$
is defined as the curl of the velocity field: $\boldsymbol{\mu} =
\nabla \times \mathbf{u}$, and the \textbf{\emph{circulation}} around the rigid
body is the line integral of $\mathbf{u}$ along any curve encircling the
rigid body.  In two dimensions, $\boldsymbol{\mu}$ can be written as
$\boldsymbol{\mu} = \mu \mathbf{e}_3$, where $\mu$ is called the scalar
vorticity.

According to Noether's theorem, there is a conserved quantity
associated to the particle relabelling symmetry.  This conserved
quantity turns out to be precisely the circulation of the fluid:
\[
  \frac{d}{dt} \oint_\mathcal{C} \mathbf{u} \cdot d \mathbf{l} = 0,
\]
and Noether's theorem hence becomes Kelvin's theorem, which states
that circulation is materially constant.  As a consequence 
of Green's theorem, the circulation of the fluid
is related to the vorticity:
\[
  \oint_\mathcal{C} \mathbf{u} \cdot d \mathbf{l} = \int_ \mathcal{S}  \nabla \times \mathbf{u}
  \cdot d \mathbf{S},
\]
where $\mathcal{S} $ is any surface whose boundary is $\mathcal{C}$.  It would lead
us too far to explore the geometry of vorticity and circulation in
detail (for this, we refer to \cite{ArnoldKhesin}) but at this stage
we just note that \emph{the conservation of vorticity is closely linked to
the particle relabelling symmetry}.

\paragraph{The Helmholtz-Hodge Decomposition.}

If (in addition to being incompressible) the fluid is irrotational,
meaning that $\nabla \times \mathbf{u} = 0$, then there exists a velocity
potential $\Phi$ such that $\mathbf{u} = \nabla \Phi$.  In the presence
of point vortices, the fluid is not irrotational but the velocity
field $\mathbf{u}$ can be uniquely decomposed in an irrotational part
and a vector field $\mathbf{u}_{\rm v}$ representing the ``rotational''
contributions (see for instance \cite{Sa1992} or \cite{Ne01}).  This is the well-known Helmholtz-Hodge decomposition:
\begin{equation} \label{hhdecomp}
  \mathbf{u} = \nabla \Phi + \mathbf{u}_{\rm v}.
\end{equation}

Here, $\mathbf{u}_{\rm v}$ is a divergence-free vector field which is
tangent to the boundary of $\mathcal{F}$, while the potential $\Phi$ is the
solution to Laplace's equation $\nabla^2 \Phi = 0$, subject to the
boundary conditions that the normal derivative of $\Phi$ equals the
normal velocity of the rigid body, and that the velocity vanishes at
infinity.  In other words, one has
\begin{equation} \label{neumann} \nabla^2 \Phi = 0 \quad \text{and} \quad 
  \frac{\partial \Phi}{\partial n} \Big|_{\partial \mathcal{F}} =
  (\boldsymbol{\omega} \times (\mathbf{x} - \mathbf{x}_0) + \mathbf{v})
  \cdot \mathbf{n} \quad \text{(for $\mathbf{x} \in \partial \mathcal{F}$)},
\end{equation}
and $\Phi$ goes to zero as $\left|\mathbf{x}\right|$ goes to infinity.
Here $\boldsymbol{\omega}$ and $\mathbf{v}$ are the angular and
translational velocity of the rigid body, expressed in a spatial
frame, while $\mathbf{x}_0$ represents the location of the center of
mass.

Since $\Phi$ depends on $\omega$ and $\mathbf{v}$ in a
linear way, we may decompose $\Phi$, following Kirchhoff, as
\begin{equation} \label{kirchhoff}
\Phi(g, \dot{g}; \mathbf{x}) = v_x \Phi_x(g, \mathbf{x}) + v_y \Phi_y(g,
\mathbf{x}) + \omega \Phi_\omega(g, \mathbf{x}).
\end{equation}
Here $\Phi_x$, $\Phi_y$, and $\Phi_\omega$ are elementary velocity
potentials corresponding to infinitesimal translations in the $x$- and
$y$-direction and to a rotation, respectively.  They satisfy the
Laplace equation with the following boundary conditions:
\[
  \frac{\partial \Phi_\omega}{\partial n} = (\mathbf{e}_3 \times
  \mathbf{x})\cdot\mathbf{n}, \quad
  \frac{\partial \Phi_x}{\partial n} = \mathbf{e}_1 \cdot \mathbf{n}, \quad \text{and} \quad 
  \frac{\partial \Phi_y}{\partial n} = \mathbf{e}_2 \cdot \mathbf{n}.
\]

Note that $\Phi$ depends on the boundary data, and hence on $(g,
\dot{g})$, while the elementary potentials depend on the location of
the rigid body (encoded by $g$) only.  For the sake of clarity, we
will suppress these arguments when no confusion is possible.  The elementary velocity potentials for a circular body are given below: the reader can check that they indeed depend on the location of the rigid body.

The Helmholtz-Hodge decomposition defines a connection on the
principal fiber bundle $Q \rightarrow \operatorname{SE}(2) $.  To see this, note that
any tangent vector to $Q$ is of the form $(g, \dot{g}; \varphi,
\dot{\varphi})$.  We define its horizontal and vertical part by
applying the Helmholtz-Hodge decomposition to $\mathbf{u} = \dot{\varphi}
\circ \varphi^{-1}$, and we put
\[
(g, \dot{g}; \varphi, \dot{\varphi})^H = (g, \dot{g}; \varphi, \nabla
\Phi \circ \varphi) \quad \text{and} \quad  (g, \dot{g}; \varphi, \dot{\varphi})^V =
(g, 0; \varphi, \mathbf{u}_{\rm v} \circ \varphi)
\]
We will verify in section~\ref{sec:neumann} below that this
prescription indeed defines a connection, which we term the
\textbf{\emph{Neumann connection}}, since its horizontal subspaces are found by
solving the Neumann problem (\ref{neumann}). This connection was used
in a variety of contexts, ranging from the dynamics of fluid drops
(see \cite{freeboundary, montgomery}) to problems in
optimal transport (see \cite{KhLe08}).

\paragraph{The Lie Algebra of Divergence-free Vector Fields and its Dual.}

At least on a formal level, $\mathrm{Diff}_{\mathrm{vol}}$ is a Lie group with
associated to it a Lie algebra, denoted by $\mathfrak{X}_{\mathrm{vol}}$ and consisting
of divergence-free vector fields which are parallel to the boundary of
$\mathcal{F}_0$.  The bracket on $\mathfrak{X}_{\mathrm{vol}}$ is the Jacobi-Lie bracket of
vector fields (which is the negative of the usual bracket of vector
fields) and its dual, denoted by $\mathfrak{X}^{\ast} _{\mathrm{vol}}$, is the set of linear
functionals on $\mathfrak{X}_{\mathrm{vol}}$.  This set can be identified with the set of
one-forms on $\mathcal{F}_0$ modulo exact forms:
\[
  \mathfrak{X}^{\ast} _{\mathrm{vol}}(\mathcal{F}_0) = \Omega^1(\mathcal{F}_0)/ \mathbf{d} \Omega^0(\mathcal{F}_0),
\] 
(see \cite{ArnoldKhesin}) while the duality pairing
between elements of $\mathfrak{X}_{\mathrm{vol}}$ and $\mathfrak{X}^{\ast} _{\mathrm{vol}}$ is given by
\[
\left< \mathbf{u}, [\alpha] \right> = 
  \int_{\mathcal{F}_0} \alpha(\mathbf{u}) \, \eta_0,
\]
where $\mathbf{u} \in \mathfrak{X}_{\mathrm{vol}}$ and $[\alpha] \in \mathfrak{X}^{\ast} _{\mathrm{vol}}$.  Note
that the right-hand side does not depend on the choice of
representative $\alpha$.  

Another interpretation of the dual Lie algebra $\mathfrak{X}^{\ast} _{\mathrm{vol}}$ is as
the set of exact two-forms on $\mathcal{F}_0$.  Any class $[\alpha]$ in
$\mathfrak{X}_{\mathrm{vol}}^{\ast} $ is uniquely determined by the exterior differential
$\mathbf{d} \alpha$ and by the value of $\alpha$ on the generators of the
first homology of $\mathcal{F}_0$.  In our case, the first homology group of
$\mathcal{F}_0$ is generated by any closed curve $\mathcal{C}$ encircling the
rigid body, and its pairing with $\alpha$ is given by
\[
   \Gamma = \int_\mathcal{C} \alpha.
\]
Under this identification, $\mathbf{d} \alpha$ represents the
\textbf{\emph{vorticity}}, while $\Gamma$ represents the
\textbf{\emph{circulation}}.  Since in our case the circulation is assumed to
be zero, it follows that $[\alpha]$ is completely determined by $\mathbf{d}
\alpha$.

From a geometric point of view, the vorticity field can be interpreted
as an element of $\mathfrak{X}_{\mathrm{vol}}^{\ast} $: if we assume that the fluid is
moving on an arbitrary Riemannian manifold, then the vorticity can be
defined by 
\[
  \mu = \mathbf{d} (\varphi^{\ast}  \mathbf{u}^\flat).
\]
Here $\flat$ is the flat operator associated to the metric.  In the
case of Euclidian spaces, this definition reduces to the one involving
the curl of $\mathbf{u}$.  This definition is slightly different from the
one in \cite{ArnoldKhesin}, where vorticity is defined as a two-form
on the inertial space $\mathbb{R}^2$, whereas in our interpretation,
vorticity lives on the material space $\mathcal{F}_0$.  Both definitions are
related by push-forward and pull-back by $\varphi$, and hence carry
the same amount of information.  Our definition has the advantage that
$\mu$ is naturally an element of $\mathfrak{X}_{\mathrm{vol}}^{\ast} $, which is preferable
from a geometric point of view.

We finish this section by noting that any Lie group acts on its Lie
algebra and its dual Lie algebra through the adjoint and the
co-adjoint action, respectively.  For the group of volume-preserving
diffeomorphisms, both are given by pull-back: if $\phi$ is an element
of $\mathrm{Diff}_{\mathrm{vol}}$, and $\mathbf{u}$ and $\mathbf{d} \alpha$ are elements of
$\mathfrak{X}_{\mathrm{vol}}$ and $\mathfrak{X}_{\mathrm{vol}}^{\ast} $, respectively, then
\[
  \mathrm{Ad}_\phi(\mathbf{u}) = \phi^{\ast}  \mathbf{u}
    \quad \text{and} \quad  
  \mathrm{CoAd}_\phi(\mathbf{d} \alpha) = \mathbf{d} (\phi^{\ast}  \alpha).
\]

\subsection{Point Vortices Interacting with a Circular Cylinder}

In this section, we impose some specific assumptions on the rigid body
and the fluid.  It should be pointed out that while these assumptions
greatly simplify the exposition, the general reduction procedure can
be carried out under far less stringent assumptions.  Later on, we
will discuss how some of these assumptions may be removed.

\paragraph{The Rigid Body.}

In order to tie in this work with previous research efforts, we assume
the rigid body to be \textbf{\emph{circular}} with radius $R$ and neutrally
buoyant (\emph{i.e.} its body weight $m$ is balanced by the force of
buoyancy).  If the density of the fluid is set to $\rho = 1$, this
implies that the body has mass $m = \pi R^2$.  The moment of inertia
of the body around the axis of symmetry is denoted by $\mathbb{I}$.

For a rigid planar body of circular shape, the elementary velocity
potentials $\Phi_x$, $\Phi_y$, and $\Phi_\omega$ occurring in
(\ref{kirchhoff}) can be calculated analytically (see \cite{lamb}) and
are given by
\begin{equation} \label{elempot}
  \Phi_x = - R\frac{x-x_0}{(x-x_0)^2 + (y-y_0)^2}, \quad \text{and} \quad 
\Phi_y = - R\frac{y-y_0}{(x-x_0)^2 + (y-y_0)^2}, 
\end{equation}
while $\Phi_\omega = 0$, reflecting the rotational symmetry of the
body.  Here $(x_0, y_0)$ are the coordinates of the center of
the disc.

In some cases, it will be more convenient to express $\Phi$ in body
coordinates.  In analogy with (\ref{kirchhoff}), we may write 
\[
  \Phi(g, \dot{g}, \mathbf{X}) = V_x \Phi_X(\mathbf{X}) + V_y
  \Phi_Y(\mathbf{X}) + \Omega \Phi_\Omega(\mathbf{X}), 
\]
where $(V_x, V_y)$ and $\Omega$ are the translational and angular
velocity in the body frame, respectively.  For the circular cylinder,
the elementary potentials in body frame are given by 
\begin{equation} \label{bodypot}
  \Phi_X = - R\frac{X}{X^2 + Y^2}, \quad 
\Phi_Y = - R\frac{Y}{X^2 + Y^2} \quad \text{and} \quad  \Phi_\Omega = 0.
\end{equation}
Note that $\Phi_X, \Phi_Y$ and $\Phi_\Omega$ do not depend on the
location of the rigid body, in contrast to $\Phi_x, \Phi_y$ and
$\Phi_\omega$.

\paragraph{Point Vortices.}

As for the fluid, we make the fundamental assumption that the
vorticity is concentrated in $N$ \textbf{\emph{point vortices}} of strengths
$\Gamma_i$, $i= 1, \ldots, N$, and that there is no circulation. 
Considered as a two-form on $\mathcal{F}_0$, the 
former means that the vorticity is given by
\begin{equation} \label{vorticity} \mu = \sum_{i = 1}^N \Gamma_i
  \delta(\bar{\mathbf{x}} - \bar{\mathbf{x}}_i) \mathbf{d} \bar{x} \wedge \mathbf{d}
  \bar{y},
\end{equation}
where $(\bar{x}, \bar{y})$ are coordinates on $\mathcal{F}_0$, and
$\bar{\mathbf{x}}_i$ is the reference location of the $i$th vortex, $i =
1, \ldots, N$.  As pointed out above, $\mu$ is an element of $\mathfrak{X}_{\mathrm{vol}}^{\ast} $.

\paragraph{The Kirchhoff-Routh Function.}

\cite{cylvortices} showed that the kinetic energy for the vortex
system is the negative of the Kirchhoff-Routh function $W_G$ for a
system of $N$ point vortices moving in a domain with moving
boundaries:
\begin{equation} \label{KR}
  T_{\mathrm{vortex}}(\mathbf{X}_1, \ldots, \mathbf{X}_N) = -W_G(\mathbf{X}_1, \ldots,
  \mathbf{X}_N). 
\end{equation}

The precise form of $W_G$ is given by
\begin{equation} \label{krfunction}
W_G(\mathbf{X}_1, \ldots, \mathbf{X}_N) = \sum_{i,j (i>j)} \Gamma_i \Gamma_j
G(\mathbf{X}_i, \mathbf{X}_j) + \frac{1}{2} \sum_i \Gamma_i g(\mathbf{X}_i,
\mathbf{X}_i),
\end{equation}
where $G(\mathbf{X}_0, \mathbf{X}_1)$ is a Green's function for the Laplace
operator of the form 
\begin{equation} \label{green}
  G(\mathbf{X}_0, \mathbf{X}_1) = g(\mathbf{X}_0, \mathbf{X}_1) + \frac{1}{4\pi}
  \log\left\Vert\mathbf{X}_0 - \mathbf{X}_1\right\Vert^2.
\end{equation}
The function $g$ is harmonic in the fluid domain and is the stream
function of $\mathbf{u}_I$.  As pointed out in \cite{cylvortices}, in the
case of a circular cylinder, $g$ can be calculated explicitly using
Milne-Thomson's circle theorem (\cite{MiTh1968}), and is given by
\[
  g(\mathbf{X}, \mathbf{Y}) = \frac{1}{4\pi} \log \left\Vert\mathbf{X}\right\Vert^2 -
  \frac{1}{4\pi} \log\left\Vert\mathbf{X} - \frac{R^2}{\left\Vert\mathbf{Y}\right\Vert^2} \mathbf{Y}\right\Vert^2.
\]

From a geometrical point of view, \cite{mw1983} showed that the vortex energy can be obtained (up to some ``self-energy'' terms) by inverting the relation $\mu = \nabla \times \mathbf{u}$, where $\mu$ is the scalar part of (\ref{vorticity}), and substituting the resulting velocity field $\mathbf{u}$ generated by $N$ vortices into the expression for the kinetic energy of the fluid.  Although the analysis of \cite{mw1983} was for  an unbounded fluid domain, their result can easily be extended to the case considered here, yielding again the negative of the Kirchhoff-Routh function as in (\ref{KR}).

\paragraph{Dynamics of the Fluid-Solid System.}

As stated in the introduction, \cite{cylvortices} (SMBK) were the
first to derive the equations of motion for a rigid cylinder
interacting with $N$ point vortices.  These equations 
generalize both the Kirchhoff equations for a rigid body in a
potential flow and the equations for $N$ point vortices in a bounded
flow.  Rather remarkably, SMBK established by direct inspection that these equations are
Hamiltonian with respect to the canonical Poisson structure on $\mathfrak{se}(2)^{\ast} 
\times \mathbb{R}^{2N}$ (\emph{i.e.} the sum of the Poisson structures on both
factors), and a Hamiltonian given below involving the kinetic energy and interaction terms.

The SMBK equations are given by
\begin{equation} \label{SMBK}
\frac{d \mathbf{L}}{d t} = 0, \quad \frac{d \mathbf{A}}{d t} + \mathbf{V}
\times \mathbf{L} = 0, \quad \text{and} \quad  \Gamma_k \frac{d \mathbf{X}_k}{dt} = -J
\frac{\partial H}{\partial \mathbf{X}_k},
\end{equation}
where $\mathbf{L}$ and $\mathbf{A} = A \mathbf{e}_3$ are the translational and
angular momenta of the system, defined by
\begin{align}
  \mathbf{L} & = c \mathbf{V} + \sum_{k=1}^N \Gamma_k \mathbf{X}_k \times \mathbf{e}_3 +
  \sum_{k=1}^N \Gamma_k \mathbf{e}_3 \times
  \frac{\mathbf{X}_k}{\left\Vert\mathbf{X}_k\right\Vert^2} \label{shiftedmomentum} \\
  A & = \mathbb{I} \Omega - \frac{1}{2} \sum_{i=1}^N \Gamma_i \left\Vert\mathbf{X}_i\right\Vert^2, \nonumber
\end{align}
and $H$ is the Hamiltonian: 
\begin{align}
H(\mathbf{L}, \mathbf{X}_k) & = -W(\mathbf{L}, \mathbf{X}_k) + \frac{1}{2c}
\left\Vert\mathbf{L}\right\Vert^2  - \frac{1}{c} \Big( \sum \Gamma_k (\mathbf{L} \times
  \mathbf{X}_k) \cdot \mathbf{e}_3 \nonumber \\ &- \frac{1}{2} \sum \Gamma_k^2
  \left\Vert\mathbf{X}_k\right\Vert^2 - \sum_{j > k} \Gamma_k\Gamma_j
  \mathbf{X}_k\cdot\mathbf{X}_j \label{bighamilt} \\ & + \frac{1}{2} \left<\sum \Gamma_k
    \frac{\mathbf{X}_k}{\left\Vert\mathbf{X}_k\right\Vert^2}, \sum \Gamma_k
    \frac{\mathbf{X}_k}{\left\Vert\mathbf{X}_k\right\Vert^2}\right>\Big). \nonumber
\end{align}

Here, $c$ is the total mass of the cylinder, consisting of the
intrinsic mass $m$ and the added mass $\pi R^2$ (due to the presence
of the fluid): $c = m + \pi R^2 = 2\pi R^2$.

\begin{theorem}
  The SMBK equations are Hamiltonian on the space $\mathfrak{se}(2)^{\ast} 
  \times \mathbb{R}^{2N}$ equipped with the Poisson bracket
  \begin{equation} \label{PoissSMBK}
    \{F, G\}_{\mathfrak{se}(2)^{\ast}  \times \mathbb{R}^{2N}} =
    \{F_{|\mathfrak{se}(2)^{\ast} }, G_{|\mathfrak{se}(2)^{\ast} }\}_{\mathrm{l.p}} + 
    \{F_{|\mathbb{R}^{2N}}, G_{|\mathbb{R}^{2N}}\}_{\mathrm{vortex}}.
  \end{equation}
  Here, $F$ and $G$ are functions on $\mathfrak{se}(2)^{\ast}  \times
  \mathbb{R}^{2N}$.  
\end{theorem}

In the theorem above, $\{\cdot, \cdot\}_{\mathrm{l.p}}$ is the Lie-Poisson
bracket on $\mathfrak{se}(2)^{\ast} $: 
\begin{equation} \label{liepoisson}
  \{f_1, f_2\}_{\mathfrak{se}(2)^{\ast} } = (\nabla f_1)^T \Lambda \nabla f_2,
\quad 
\text{where}
\quad
  \Lambda = \begin{pmatrix} 
    0 & -P_y & P_x \\
    P_y & 0 & 0 \\
    -P_x & 0 & 0 
    \end{pmatrix},
\end{equation}
for arbitrary functions $f_1, f_2$ on $\mathfrak{se}(2)^{\ast} $.  Similarly,
$\{\cdot, \cdot\}_{\mathrm{vortex}}$ is the vortex bracket, given by:
\begin{equation} \label{vortexbracket}
  \{g_1, g_2\}_{\mathrm{vortex}} = \sum_{i=1}^N \frac{1}{\Gamma_i} \left( \frac{\partial
    g_1}{\partial X_i}\frac{\partial
    g_2}{\partial Y_i} - \frac{\partial
    g_2}{\partial X_i}\frac{\partial
    g_1}{\partial Y_i} \right),
\end{equation}
where $g_1, g_2$ are arbitrary functions on $\mathbb{R}^{2N}$.

\paragraph{The BMR Equations.}

A completely different perspective on the rigid body interacting with
$N$ point vortices is offered by \cite{BoMaRa2003} (BMR).  From their
point of view, the equations of motion are again written in
Hamiltonian form on $\mathfrak{se}(2)^{\ast}  \times \mathbb{R}^{2N}$, but now with a
\emph{noncanonical} Poisson bracket:
\[
  \dot{F} = \{F, H\}_{\mathrm{BMR}}
\]
for all functions $F$ on $\mathfrak{se}(2)^{\ast}  \times \mathbb{R}^{2N}$.  The
Hamiltonian is the sum of the kinetic energies of the subsystems,
without interaction terms:
\[
  H = \frac{c}{2} \left< \mathbf{V}, \mathbf{V} \right> -
  W_G(\mathbf{V}, \mathbf{X}_k),
\]
whereas the Poisson bracket is determined by 
its value on the coordinate functions:
\begin{align}
  \{V_1, V_2\}_{\mathrm{BMR}} & = \frac{\Gamma}{c^2} - \sum \frac{\Gamma}{c^2}
  \frac{R_i^4 - R^4}{R_i^4}, \, & \{V_1, X_i\}_{\mathrm{BMR}} & = \frac{1}{c} 
  \frac{R_i^4 - R^2(X_i^2 - Y_i^2)}{R_i^4},\nonumber \\
  \{V_1, Y_i\}_{\mathrm{BMR}} & = -\frac{2R^2}{c} \frac{X_iY_i}{R_i^4}, 
  & \{V_2, X_i\}_{\mathrm{BMR}} & = -\frac{2R^2}{c}
  \frac{X_iY_i}{R_i^4}, \label{bmrbracket} \\
  \{V_2, Y_i\}_{\mathrm{BMR}} & =  \frac{1}{c} \frac{R_i^4 + R^2(X_i^2 -
    Y_i^2)}{R_i^4}, & \{X_i, Y_i\}_{\mathrm{BMR}} & = - \frac{1}{\Gamma_i}, \nonumber
\end{align}
where $R_i^2 = \left\Vert\mathbf{X}_i\right\Vert^2$, $\mathbf{V} = (V_1, V_2)$, 
$\mathbf{X}_i = (X_i, Y_i)$, and $\Gamma$ is the total vortex strength: 
\[
   \Gamma = \sum_{i=1}^N \Gamma_i.
\]
Note that this Poisson bracket differs from the one in
\cite{BoMaRa2003} by an overall factor of $2\pi$.  \cite{Sh2005}
showed that this discrepancy can be attributed to the way in which BMR
choose the fluid density.

\paragraph{The Link between the SMBK and the BMR Equations.}

By explicit calculation, \cite{Sh2005} constructed a Poisson 
map taking the SMBK equations into the BMR equations.  His
result is listed below.

\begin{theorem} \label{thm:shashi}
  The map $\mathcal{S} : (\mathbf{L}, A; \mathbf{X}_1, \ldots, \mathbf{X}_N) \mapsto
  (\mathbf{V}, \Omega; \mathbf{X}_1, \ldots, \mathbf{X}_N)$, where $\mathbf{L}$
  and $A$ are given by \textup{\eqref{shiftedmomentum}}, is a Poisson map from
  $\mathfrak{se}(2)^{\ast}  \times \mathbb{R}^{2N}$ equipped with the bracket $\{\cdot,
  \cdot\}_{\mathfrak{se}(2)^{\ast}  \times \mathbb{R}^{2N}}$ to 
  $\mathfrak{se}(2)^{\ast}  \times \mathbb{R}^{2N}$ with the bracket $\{\cdot,
  \cdot\}_{\mathrm{BMR}}$.  
\end{theorem}

Even though this result asserts that both sets of equations are
equivalent, it leaves open the question as to \emph{why} this is so.
By re-deriving the equations of motion using symplectic reduction, not
only do we obtain both sets of equations, but the map $\mathcal{S} $ also
follows naturally.

\section{The Neumann Connection} \label{sec:neumann}

The bundle $\pi: Q \rightarrow \operatorname{SE}(2) $ is equipped with a principal
fibre bundle connection, called the \textbf{\emph{Neumann connection}} by
\cite{MontgomeryThesis}.  There are many ways of describing this
connection, but from a physical point of view, the definition using
the horizontal lift operator is perhaps most appealing.  From this point
of view, the Neumann connection is a map
\begin{equation} \label{holift}
   \mathbf{h}: Q \times T\operatorname{SE}(2)  \rightarrow T Q, \quad 
   \mathbf{h}(\varphi, g, \dot{g}) = \nabla \Phi \circ \varphi, 
\end{equation}
where $\Phi$ is the solution of the Neumann problem (\ref{neumann})
associated to $(g, \dot{g})$.  In other words, the Neumann connection
associates to each motion $(g, \dot{g})$ the corresponding induced
velocity field of the fluid, and hence encodes the effect of the body
on the fluid.  It is important to note that the Neumann connection
does not depend on the point vortex model and is valid for any
vorticity field.

Similar connections as this one have been described before (see for
example \cite{freeboundary, KhLe08}) but a complete overview of its
definition and properties seems to be lacking.  In this section, we
give an outline of the properties of the Neumann connection which are
relevant for the developments in this paper, leaving detailed proofs
for the appendix.

\paragraph{Invariance of the Kirchhoff Decomposition.}

Before introducing the Neumann connection, we prove that the velocity
potential $\Phi(g, \dot{g}; \mathbf{x})$ is left $\operatorname{SE}(2) $-invariant,
expressing the fact that the dynamics is invariant under translations
and rotations of the combined solid-fluid system.

\begin{proposition} \label{prop:invariance}
The velocity potential $\Phi$ is left $\operatorname{SE}(2) $-invariant
in the sense that 
\begin{equation} \label{invariance}
\Phi(hg, TL_h(\dot{g}); h\mathbf{x}) = \Phi(g, \dot{g}; \mathbf{x})
\end{equation}
for all $h \in \operatorname{SE}(2) $ and $(g, \dot{g}; \mathbf{x}) \in T \operatorname{SE}(2)  \times
\mathbb{R}^2$. 
\end{proposition}
\begin{proof}
  This assertion can be proved in a number of different ways.  The
  easiest is to use the assumption that the body is circular and solve
  the equation for the elementary potentials explicitly.  Recall that
  these elementary potentials are given by (\ref{elempot}).  It is
  then straightforward to check that (\ref{invariance}) holds, using
  the transformation properties of the velocity in the inertial frame.
\end{proof}

\paragraph{The Connection One-form.}

For our purposes, it is convenient to define the Neumann connection
through its connection one-form $\mathcal{A}: TQ \rightarrow \mathfrak{X}_{\mathrm{vol}}$ given
by
\begin{equation} \label{connection}
  \mathcal{A}_{(\varphi, g)}(\dot{g}, \dot{\varphi}) = \varphi^{\ast} 
  \mathbf{u}_{\rm v},
\end{equation}
where $\mathbf{u}_{\rm v}$ is the divergence-free part in the
Helmholtz-Hodge decomposition (\ref{hhdecomp}) of the Eulerian
velocity $\mathbf{u} = \dot{\varphi} \circ \varphi^{-1}$.  A proof that
this prescription determines a well-defined connection form can be
found in the appendix, proposition~\ref{prop:conn}, where it is also shown that this prescription agrees with the horizontal lift operator (\ref{holift}).

\paragraph{The Curvature of the Neumann Connection.}

It will be convenient in what follows to have an expression for the
$\mu$-component of the curvature of the Neumann connection, where for
now $\mu$ is an arbitrary element of $\mathfrak{X}_{\mathrm{vol}}^{\ast} $.  Later on, $\mu$
will be the vorticity (\ref{vorticity}) associated to $N$ point
vortices.

The curvature of a principal fiber bundle connection is a two-form $\mathcal{B}$ whose definition is listed in the appendix.  For the  
Neumann connection we have
\begin{equation} \label{curv}
  \mathcal{B}_{(g, \varphi)}((\dot{g}_1, \dot{\varphi}_1), (
  \dot{g}_2, \dot{\varphi}_2)) = - \mathcal{A}_{(g, \varphi)}([X^H, Y^H]),
\end{equation}
where $X^H$ and $Y^H$ are horizontal vector fields on $Q$ such that
$X^H(g, \varphi) = (\dot{g}_1, \dot{\varphi}_1)$ and $Y^H(g, \varphi)
= (\dot{g}_2, \dot{\varphi}_2)$.   

\begin{proposition} \label{prop:curv} Let $(\dot{g}_1,
  \dot{\varphi}_1)$ and $(\dot{g}_2, \dot{\varphi}_2)$ be elements of
  $T_{(g, \varphi)} Q$ and denote the solutions of the Neumann problem
  (\ref{neumann}) associated to $(g, \dot{g}_1)$ resp. $(g,
  \dot{g}_2)$ by $\Phi_1$ and $\Phi_2$.  Then the $\mu$-component of
  the curvature $\mathcal{B}$ is given by
  \begin{equation} \label{curvature} \left<\mu, \mathcal{B}_{(g,
        \varphi)}((\dot{g}_1, \dot{\varphi}_1), (\dot{g}_2,
      \dot{\varphi}_2))\right> = \left\langle \! \left\langle\mu, \mathbf{d} \Phi_1 \wedge \mathbf{d}
      \Phi_2\right\rangle \!
  \right\rangle_{\mathcal{F}}-\int_{\partial \mathcal{F}} \alpha \wedge \ast(\mathbf{d}
    \Phi_1 \wedge \mathbf{d} \Phi_2),
  \end{equation}
  where $\left\langle \! \left\langle\cdot, \cdot\right\rangle \!
  \right\rangle_\mathcal{F}$ is the metric on the space of
  forms on $\mathcal{F}$, induced by the Euclidian metric on $\mathcal{F}$.
\end{proposition}

\begin{proof} Let $\mu$ be equal to $\varphi^{\ast}  \mathbf{d} \alpha$, and
  pick $\alpha$ such that $\alpha = \mathbf{u}_{\rm v}^\flat$, where
  $\mathbf{u}_{\rm v}$ is a divergence-free vector field on $\mathcal{F}$
  tangent to $\partial \mathcal{F}$.  For example, in the case of the
  vorticity due to $N$ point vortices, $\mathbf{u}_{\rm v}$ is the
  velocity field due to the vortices and their images (see
  \cite{Sa1992}).  The calculation of the curvature involves computing
  the Jacobi-Lie bracket of two horizontal vector fields and taking
  the divergence-free part of the result.  Because of the special form
  of $\alpha$ we can dispense with the latter step, since $\alpha$ is
  chosen to be $L_2$-orthogonal to gradient vector fields.  Therefore,
  the $\mu$-component of the curvature is given by
  \begin{align*}
    \left<\mu, \mathcal{B}_{(g, \varphi)}((\dot{g}_1, \dot{\varphi}_1),
      (\dot{g}_2, \dot{\varphi}_2))\right>  & = - \left<\mu, \mathcal{A}_{(g,
        \varphi)}([X^H, Y^H])\right>
    \\
    & = - \int_{\mathcal{F}_0} (\varphi^{\ast}  \alpha)([X^H, Y^H]) \, \eta_0,
  \end{align*}
where the bracket on the left-hand side is the Jacobi-Lie bracket,
which is the negative of the usual commutator of vector fields.  

The bracket can be made more explicit by noting that (as vector fields
on $Q$)
\[
  [X^H, Y^H](g, \varphi) = ([(\mathbf{d} \Phi_1)^\flat, (\mathbf{d} \Phi_2)^\flat] \circ
  \varphi, \cdots),
\]
where the dots denote a term in $T_g \operatorname{SE}(2) $ whose explicit form
doesn't matter.  The curvature then becomes
\[
\left<\mu, \mathcal{B}_{(g, \varphi)}((\dot{g}_1, \dot{\varphi}_1),
  (\dot{g}_2, \dot{\varphi}_2))\right> = - \int_\mathcal{F} \alpha(
[(\mathbf{d}\Phi_1)^\flat, (\mathbf{d}\Phi_2)^\flat]) \, d \mathbf{x}.
\]

The remainder of the proof relies on the following formula for the
codifferential of a wedge product (\cite[formula~1.34]{Vaisman}):
\begin{equation} \label{codiff}
  \delta (\alpha \wedge \beta) = \delta \alpha \wedge \beta +
    (-1)^p \alpha \wedge \delta \beta - [\alpha^\flat,
    \beta^\flat]^\sharp, 
\end{equation}
where $p = \deg \alpha$.  Applying (\ref{codiff}) with $\alpha = \mathbf{d}
\Phi_1$ and $\beta = \mathbf{d} \Phi_2$ gives
\[
  [(\mathbf{d}\Phi_1)^\flat, (\mathbf{d} \Phi_2)^\flat]^\sharp = 
    \Delta \Phi_1 \mathbf{d} \Phi_2 - \Delta \Phi_2 \mathbf{d} \Phi_1 - 
    \delta(\mathbf{d} \Phi_1 \wedge \mathbf{d} \Phi_2), 
\]
where $\Delta$ is the Laplace-Beltrami operator on differential
forms.  Since both $\Phi_1$ and $\Phi_2$ are harmonic, the first two
terms of the right-hand side vanish.  Consequently, the
$\mu$-component of the curvature becomes
\begin{align*}
  \left<\mu, \mathcal{B}_{(g, \varphi)}((\dot{g}_1, \dot{\varphi}_1),
    (\dot{g}_2, \dot{\varphi}_2))\right> & = \left\langle \! \left\langle\alpha, \delta(\mathbf{d} \Phi_1
    \wedge \mathbf{d} \Phi_2)\right\rangle \!
  \right\rangle_{\mathcal{F}} \\ & = \left\langle \! \left\langle\mathbf{d} \alpha, \mathbf{d} \Phi_1 \wedge \mathbf{d}
    \Phi_2\right\rangle \!
  \right\rangle_{\mathcal{F}} - \int_{\partial \mathcal{F}} \alpha \wedge \ast(\mathbf{d} \Phi_1
  \wedge \mathbf{d} \Phi_2),
\end{align*}
using the adjointness of $\mathbf{d}$ and $\delta$ for manifolds with boundary (this is a simple consequence of Stokes' theorem).
\end{proof}

In traditional fluid mechanics notation, formula (\ref{curvature})
becomes
\[
  \left<\mu, \mathcal{B}\right> = \int_\mathcal{F} (\nabla \times \mathbf{u}_V)
  \cdot (\nabla \Phi_1 \times \nabla \Phi_2) d x 
    - \int_{\partial \mathcal{F}} \mathbf{u}_V \cdot (\mathbf{n} \times (\nabla \Phi_1
    \times \nabla \Phi_2)) d A,
\]
a formula first derived in \cite{MontgomeryThesis}.

\section{Reduction with Respect to the Diffeomorphism
  Group} \label{sec:diff} 

As mentioned in the introduction, the fluid-solid system on the space
$Q \subset \operatorname{SE}(2)  \times \operatorname{Emb}(\mathcal{F}_0, \mathbb{R}^2)$ is invariant under the
particle relabelling group $\mathrm{Diff}_{\mathrm{vol}}$.  We now perform symplectic reduction to eliminate that symmetry.  In order to do so, we need to fix a value of the momentum map associated to the $\mathrm{Diff}_{\mathrm{vol}}$-symmetry.  From a physical point of view, this boils down to fixing the vorticity of the system; it is at this point that the assumption is used that the vorticity is concentrated in $N$ point vortices.

Before tackling symplectic reduction in the context of the fluid-solid system, we
first give a general overview of cotangent bundle reduction following
\cite{MarsdenHamRed}.  Roughly speaking, applying symplectic reduction
to a cotangent bundle yields a space which is diffeomorphic to the
product of a reduced cotangent bundle and a co-adjoint orbit of the
group.  The dynamics on the reduced space is governed by a reduced
Hamiltonian and a magnetic symplectic form: the reduced symplectic
form is the sum of the canonical symplectic forms on the individual
factors and an additional magnetic term.

As we shall see below, in the case of the fluid-solid system, the
reduced phase space is equal to $T^{\ast}  \operatorname{SE}(2)  \times \mathbb{R}^{2N}$, where the
first factor describes the rigid body, while the second factor
determines the configuration of the vortices.  At first sight, it may
therefore appear that the intermediate fluid is completely gone.  Yet,
the vortices act on the rigid body, and vice versa, 
through the
surrounding fluid.  The answer to this apparent
contradiction is that the effect of the fluid is concentrated in the
magnetic symplectic form, for which we derive a convenient expression
below.

\subsection{Cotangent Bundle Reduction: Review}

In this section, we collect some relevant results from
\cite{MarsdenHamRed}. We consider a manifold $Q$ on which a Lie group
$G$, with Lie algebra $\mathfrak{g}$, acts from the right and we denote the
action by $\sigma: Q \times G \rightarrow Q$.  In addition, we assume
that we are given a connection one-form $\mathcal{A} : TQ \rightarrow
\mathfrak{g}$ with curvature two-form $\mathcal{B}$.  In the rest of this paper, $Q$ will be the configuration
space of the solid-fluid system, while the structure group $G$ will be
$\mathrm{Diff}_{\mathrm{vol}}$, and the connection $\mathcal{A}$ the Neumann connection.

\paragraph{The Curvature as a Two-form on the Reduced Space.}

Let $\mu$ be an element of $\mathfrak{g}^{\ast} $, and denote its isotropy
subgroup under the co-adjoint action by $G_\mu$:
\[
  G_\mu = \left\{ g \in G : \mathrm{CoAd}_g \mu = \mu \right\}.
\]
Consider the contraction $\left<\mu, \mathcal{B}\right>$ of $\mu$ with the
curvature $\mathcal{B}$: due to the $G$-equivariance and the fact that
$\mathcal{B}$ vanishes on vertical vectors (see (\ref{props})), we may show
that $\left<\mu, \mathcal{B}\right>$ is a $G_\mu$-invariant form.

\begin{proposition}
  The $\mu$-component of the curvature $\mathcal{B}$ is a $G_\mu$-invariant
  two-form:
  \begin{enumerate}
    \item $\sigma^{\ast} _g \left<\mu, \mathcal{B}\right> = \left<\mu, \mathcal{B}\right>$, for all
      $g \in G_\mu$; 
    \item $i_{\xi_Q} \left<\mu, \mathcal{B}\right> = 0$ for all $\xi \in \mathfrak{g}_\mu$.
  \end{enumerate}
  Hence, $\left<\mu, \mathcal{B}\right>$ induces a two-form $\bar{\mathcal{B}}_\mu$ on
  $Q/G_\mu$ such that 
  \[
    \pi_{Q, G_\mu}^{\ast}  \bar{\mathcal{B}}_\mu = \left<\mu, \mathcal{B}\right>,
  \]
where $\pi_{Q, G_\mu} : Q \rightarrow Q/G_\mu$ is the quotient map.
\end{proposition}
\begin{proof}
For the first property, we have 
\[
  \sigma^{\ast} _g \left<\mu, \mathcal{B}\right>  = \left<\mu, \sigma_g^{\ast}  \mathcal{B}\right>
      = \left<\mu, \mathrm{Ad}_{g^{-1}}\mathcal{B}\right>
      = \left<\mathrm{CoAd}_g(\mu), \mathcal{B}\right> 
      = \left<\mu, \mathcal{B}\right>
\]
if $g \in G_\mu$.  The second item follows from the corresponding
property for $\mathcal{B}$.  Actually, more is true: $i_{\xi_Q} \left<\mu,
  \mathcal{B}\right> = 0$ for all $\xi \in \mathfrak{g}$, not just in $\mathfrak{g}_\mu$.
\end{proof}

\paragraph{Cotangent Bundle Reduction.}

The first stage in reducing the phase space consists of dividing out
the particle relabelling symmetry.  To this end, we use the framework
of \textbf{\emph{cotangent bundle reduction}} to construct the reduced phase
space.  The framework for cotangent bundle reduction outlined in \cite{MarsdenHamRed} allows us to write down the
reduced phase space and the modified symplectic form.  We quote: 

\begin{theorem} \label{thm:cotred} (Marsden \emph{et al.} [theorem~2.2.1])
\begin{enumerate} 
\item There exists a symplectic imbedding $\varphi_\mu$ of the reduced
  phase space $(T^{\ast}  Q)_\mu$ into the cotangent bundle $T^{\ast} 
  (Q/G_\mu)$ with the shifted symplectic structure $\Omega_\mathcal{B} :=
  \Omega_{\mathrm{can}} - B_\mu$.

\item The image of $\varphi_\mu$ is the vector subbundle $[T \pi_{Q,
    G_\mu}(V)]^\circ$ of $T^{\ast}  (Q/G_\mu)$, where $V \subset TQ$ is the
  vector subbundle consisting of vectors tangent to the $G$-orbits in
  $Q$, and $^\circ$ denotes the annihilator relative to the natural
  duality pairing between $T(Q/G_\mu)$ and $T^{\ast} ( Q/G_\mu)$.
\end{enumerate}
\end{theorem}

Here $(T^{\ast}  Q)_\mu$ is the reduced phase space $J^{-1}(\mu)/G_\mu$,
where $G_\mu$ is the isotropy subgroup of $\mu$.  The two-form $B_\mu$
is called the \textbf{\emph{magnetic two-form}}, and is defined as $B_\mu =
\pi_{Q/G_\mu}^{\ast}  \beta_\mu$, where $\pi_{Q/G_\mu}: T^{\ast} (Q/G_\mu)
\rightarrow Q/G_\mu$ is the cotangent bundle projection and
$\beta_\mu$ is a two-form on $Q/G_\mu$.

To define $\beta_\mu$, consider the one-form $\alpha_\mu := \left<
  \mu, \mathcal{A} \right>$ on $Q$.  One can show that $\mathbf{d} \alpha_\mu$ (but
not $\alpha_\mu$ itself) is $G_\mu$-invariant, and the induced
two-form on $Q/G_\mu$ is precisely $\beta_\mu$, or in other words
\[
  \pi^{\ast} _{Q, G_\mu} \beta_\mu = \mathbf{d} \alpha_\mu,
\]
where $\pi_{Q, G_\mu}: Q \rightarrow Q/G_\mu$ is the quotient map.
The calculation of $\mathbf{d} \alpha_\mu$ is greatly facilitated by using the
Cartan structure equation (\cite{KN1}), which allows us to rewrite $\mathbf{d}
\alpha_\mu$ as 
\begin{equation} \label{cartan} \mathbf{d} \alpha_\mu = \left< \mu, \mathcal{B}
  \right> - \left< \mu, [\mathcal{A}, \mathcal{A}] \right>.
\end{equation}
Observe the sign difference on the right-hand side with
\cite{MarsdenHamRed}, which is due to the fact that we take $G$ to be
acting on $Q$ from the \emph{right}.

\subsection{The Reduced Phase Space}

Theorem~\ref{thm:cotred} provides us with an explicit prescription for
the reduced phase space.  Recall that the diffeomorphism group
$\mathrm{Diff}_{\mathrm{vol}}$ acts on the dual Lie algebra $\mathbf{d} \Omega^1(\mathcal{F}_0)$
through pull-back.  In particular, if $\mu$ represents the vorticity
due to $N$ point vortices as in (\ref{vorticity}), then
\[
  \phi^{\ast}  \mu = \sum_{i = 1}^N \Gamma_i \delta(\mathbf{\bar{x}} -
  \phi^{-1}(\mathbf{\bar{x}}_i)) \, \mathbf{d} \bar{x} \wedge \mathbf{d} \bar{y}.
\]
In other words, the diffeomorphism group acts on the space of point
vortices by simply moving the vortices.  It also follows that
the isotropy subgroup of $\mu$ consists of all diffeomorphisms $\phi$
for which the vortex reference locations are fixed:
$\phi(\mathbf{\bar{x}}_i) = \mathbf{\bar{x}}_i$, for $i = 1, \ldots, N$.  We denote
the group of all such diffeomorphisms as $\mathrm{Diff}_{{\mathrm{vol}}, \mu}$.

The group $\mathrm{Diff}_{{\mathrm{vol}}, \mu}$ acts on $Q$ and moreover, the quotient
of $Q$ by this action is diffeomorphic to $\operatorname{SE}(2)  \times \mathbb{R}^{2N}$.
To see this, note that $\mathrm{Diff}_{{\mathrm{vol}}, \mu}$ acts on the $\operatorname{Emb}(\mathcal{F}_0,
\mathbb{R}^2)$ factor only, and that there exists a diffeomorphism of the
quotient space $\operatorname{Emb}(\mathcal{F}_0, \mathbb{R}^2)/\mathrm{Diff}_{{\mathrm{vol}}, \mu}$ with
$\mathbb{R}^{2N}$, given by
\[
  [\varphi] \mapsto (\varphi(\mathbf{\bar{x}}_1), \ldots, \varphi(\mathbf{\bar{x}}_N)).
\]

Similarly, the projection of $Q$ onto the quotient space
$Q/\mathrm{Diff}_{{\mathrm{vol}}, \mu} = \operatorname{SE}(2)  \times \mathbb{R}^{2N}$ is given by 
\[
  \pi_{Q, G_\mu}: (g, \varphi) \mapsto (g; \varphi(\mathbf{\bar{x}}_1), \ldots,
  \varphi(\mathbf{\bar{x}}_N)).
\]

\begin{proposition}
  After reducing by the group of volume preserving diffeomorphisms,
  the reduced phase is given by $T^{\ast}  \operatorname{SE}(2)  \times \mathbb{R}^{2N}$.
\end{proposition}
\begin{proof}
  The vertical bundle $V$ on $Q$ consists of vectors $(g, 0; \varphi,
  \dot{\varphi})$.  Projecting this bundle down under $T\pi_{Q,
    G_\mu}$ shows that $T\pi_{Q, G_\mu}(V)$ is spanned by elements of
  the form
\[
T\pi_{Q, G_\mu}(g, 0; \varphi, \dot{\varphi}) = (g, 0; \varphi(\mathbf{\bar{x}}_1),
\dot{\varphi}(\mathbf{\bar{x}}_1); \ldots; \varphi(\mathbf{\bar{x}}_N),
\dot{\varphi}(\mathbf{\bar{x}}_N))
\]
and as each of the $N$ vectors on the right hand side can range over
the whole of $T\mathbb{R}^2$, $T\pi_{Q, G_\mu}(V)$ is equal to $\operatorname{SE}(2)  \times
T \mathbb{R}^{2N}$.  Its annihilator is therefore $T^{\ast}  \operatorname{SE}(2)  \times
\mathbb{R}^{2N}$.
\end{proof}

The reduced symplectic structure on $T^{\ast}  \operatorname{SE}(2)  \times \mathbb{R}^{2N}$ is
described in theorem~\ref{thm:cotred}.  Explicitly, it is given by
\begin{equation} \label{magsymp}
  \Omega_\mathcal{B} = \Omega_{\mathrm{can}} - B_\mu, 
\end{equation}
where $\Omega_{\mathrm{can}}$ is the canonical symplectic structure on $T^{\ast} 
\operatorname{SE}(2) $ and $B_\mu$ is the pullback to $ T^{\ast}  \operatorname{SE}(2)  \times \mathbb{R}^{2N}$
of the form $\beta_\mu$ on $\operatorname{SE}(2)  \times \mathbb{R}^{2N}$.  From now on, we
will no longer make any notational distinction between $\beta_\mu$ and
its pull-back $B_\mu$ and denote both by $\beta_\mu$.

Apart from the reduced symplectic form, which will be determined
explicitly later on, the dynamics on the reduced space is also
governed by a reduced version of the Hamiltonian (\ref{totalenergy}).
The calculation of this Hamiltonian is the subject of the next
section.

\subsection{The Reduced Hamiltonian}

The kinetic energy (\ref{kinenergy}) is invariant under the action of
$\mathrm{Diff}_{{\mathrm{vol}}, \mu}$ on $Q$ and hence determines a reduced kinetic
energy function on the quotient space $\operatorname{SE}(2)  \times \mathbb{R}^{2N}$.  Parts
of the computation of the explicit form of the reduced kinetic energy
can be found throughout the literature, but since no single reference
has a complete picture, we briefly recall these results here.

Using the Helmholtz-Hodge decomposition (\ref{hhdecomp}), the kinetic
energy (\ref{kinenergy}) of the combined solid-fluid system can be written as
\begin{align}
  T(\dot{g}, \dot{\varphi}) & = \frac{1}{2} \left\langle \! \left\langle\dot{g},
    \dot{g}\right\rangle \!
  \right\rangle_{\operatorname{SE}(2) } +
  \frac{1}{2} \int_{\mathcal{F}} \mathbf{u} \cdot \mathbf{u} \, d\mathbf{x}  \nonumber \\
  & = \frac{1}{2}  \left\langle \! \left\langle\dot{g},\dot{g}\right\rangle \!
  \right\rangle_{\operatorname{SE}(2) } +  \frac{1}{2} \int_{\mathcal{F}} \mathbf{u}_V \cdot \mathbf{u}_V \, d\mathbf{x} +
  \frac{1}{2} \int_{\mathcal{F}} \nabla \Phi \cdot \nabla \Phi \, d\mathbf{x}, \label{redkin}
\end{align}
where we have used the fact that $\mathbf{u}_V$ and $\nabla \Phi$ are
$L_2$-orthogonal.  

Following a reasoning similar to \cite{mw1983}, it can be shown that
the kinetic energy of the vortex system is nothing but the negative of
the Kirchhoff-Routh function $W_G$ (see (\ref{krfunction})):
\[
  \frac{1}{2} \int_{\mathcal{F}} \mathbf{u}_V \cdot \mathbf{u}_V \, d \mathbf{x} 
    = - W_G(\mathbf{X}_1, \ldots, \mathbf{X}_N),
\]

The gradient term in (\ref{redkin}) can be rewritten using a standard
procedure, going back to Kirchhoff and Lamb, and yields the familiar
added masses and moments of inertia for a rigid body in potential
flow.  By using Green's theorem to rewrite the gradient term as an
integral over the boundary, and substituting the Kirchhoff expansion
(\ref{kirchhoff}), one can show that (see \cite{lamb, KMlocomotion})
\begin{equation} \label{gradterm}
\frac{1}{2} \int_{\mathcal{F}} \nabla \Phi \cdot \nabla \Phi \, d \mathbf{x} = 
  \frac{1}{2}     \begin{pmatrix} 
      \Omega & \mathbf{V}
    \end{pmatrix}
     \mathbb{M}_a
    \begin{pmatrix} 
      \Omega \\ \mathbf{V}
    \end{pmatrix},
\end{equation}
where $\mathbb{M}_a$ is a $3$-by-$3$ matrix of added masses and inertia whose
entries depend only on the geometry of the rigid body.  For the
circular cylinder, $\mathbb{M}_a$ can be evaluated explicitly: 
\[
  \mathbb{M}_a =  \begin{pmatrix}
      0 & 0 \\
      0 & \pi R^2 \mathbf{I} 
    \end{pmatrix}.
\]

The important point is to note that the gradient term (\ref{gradterm})
has the same form as the kinetic energy of the rigid body.  By
introducing the matrix $\mathbb{M} = \mathbb{M}_m + \mathbb{M}_a$, where $\mathbb{M}_m$ is the
mass matrix (\ref{Mm}), the kinetic energy of the rigid body, together
with the gradient term, can be written as
\[
  T = \frac{1}{2} \begin{pmatrix} 
      \Omega & \mathbf{V}
    \end{pmatrix}
     \mathbb{M}
    \begin{pmatrix} 
      \Omega \\ \mathbf{V}
    \end{pmatrix},
\]
which determines by left extension a function, also denoted by $T$, on $T\operatorname{SE}(2) $.
Putting everything together, we conclude that the total reduced
kinetic energy on $T\operatorname{SE}(2)  \times \mathbb{R}^{2N}$ is given by
\begin{equation} \label{totalhamiltonian}
  T_{\mathrm{red}}(\Omega, \mathbf{V}; \mathbf{X}_1, \ldots, \mathbf{X}_N) = \frac{1}{2}
\begin{pmatrix} 
      \Omega & \mathbf{V}
    \end{pmatrix}
     \mathbb{M}
    \begin{pmatrix} 
      \Omega \\ \mathbf{V}
    \end{pmatrix} - W_G(\mathbf{X}_1, \ldots, \mathbf{X}_N), 
\end{equation}
where $W_G$ is given by (\ref{krfunction}) and (\ref{green}).

Two issues are noteworthy here.  First of all, we follow \cite{mw1983}
and \cite{cylvortices} in ``regularizing'' $W_G$ by subtracting infinite
contributions that arise when putting $\mathbf{X}_0 = \mathbf{X}_1$ in the
expression for the Green's function $G$ in (\ref{green}).  Secondly, we have chosen to express the kinetic
energy in the body frame, which will be more convenient later on.

\subsection{The Magnetic Two-form $\beta_\mu$}

Recall that $\beta_\mu$ is a two-form on $\mathbb{R}^{2N} \times \operatorname{SE}(2) $
defined as
\[
\beta_\mu(\mathbf{x}_1, \ldots, \mathbf{x}_N; g)( (\mathbf{v}_1, \ldots,
\mathbf{v}_N, \dot{g}_1), (\mathbf{w}_1, \ldots, \mathbf{w}_N, \dot{g}_2)) 
= \mathbf{d}\alpha_\mu(\varphi, g)( (\dot{\varphi}_1, \dot{g}_1),
(\dot{\varphi}_2, \dot{g}_2) ),
\]
where $\mathbf{v}_i, \mathbf{w}_i \in T_{\mathbf{x}_i} \mathbb{R}^2$ ($i = 1, \ldots,
N$), and $\dot{g}_1$ and $\dot{g}_2$ are elements of $T_g \operatorname{SE}(2) $.  The
embedding $\varphi$ has to satisfy $\varphi(\mathbf{\bar{x}}_i) = \mathbf{x}_i$
for $i = 1, \ldots, N$, while $\dot{\varphi}_1$ and $\dot{\varphi}_2$
should satisfy
\[
  \dot{\varphi}_1(\mathbf{\bar{x}}_i) = \mathbf{v}_i \quad \text{and} \quad  
  \dot{\varphi}_2(\mathbf{\bar{x}}_i) = \mathbf{w}_i \quad \text{(for $i=1, \ldots, N$)}.
\]

Using Cartan's structure equation (\ref{cartan}), $\mathbf{d}\alpha_\mu$ can be
rewritten as the difference of a curvature term and a ``Lie-Poisson''
term: if we introduce two-forms $\beta_{{\mathrm{curv}}}$ and
$\beta_{{\mathrm{l.p}}}$ on $\mathbb{R}^{2N} \times \operatorname{SE}(2) $, determined by
\[
  \pi_{Q, G_\mu}^{\ast}  \beta_{{\mathrm{curv}}} = \left<\mu, \mathcal{B}\right>, \quad \text{and} \quad 
  \pi_{Q, G_\mu}^{\ast}  \beta_{{\mathrm{l.p}}} = \left<\mu, [\mathcal{A}, \mathcal{A}]\right>,
\]
then $\beta_\mu = \beta_{{\mathrm{curv}}} - \beta_{{\mathrm{l.p}}}$.

In the computation of these terms, we will frequently encounter
expressions involving the vertical part of $(\dot{\varphi}_1,
\dot{g}_1)$ and $(\dot{\varphi}_2, \dot{g}_2)$, evaluated at the
vortex locations.  We now derive a convenient expression for these
terms.

Let $\Phi_1$ be the solution of the Neumann problem (\ref{neumann})
with boundary data $\dot{g}_1$ and denote by $\mathbf{u}$ the Eulerian
velocity field associated to $\dot{\varphi}_1$.  From the
Helmholtz-Hodge decomposition $\mathbf{u} = \mathbf{u}_{\rm v} + \nabla
\Phi_1$, it follows the divergence-free part $\mathbf{u}_{\rm v}$
satisfies (for $i = 1, \ldots, N$)
\begin{align*}
\mathbf{u}_{\rm v}(\mathbf{x}_i) & = \mathbf{u}(\mathbf{x}_i) - \nabla \Phi_1(\mathbf{x}_i) \\
  & = \mathbf{v}_i - \nabla \Phi_1(\mathbf{x}_i),
\end{align*}
since $\mathbf{u}(\mathbf{x}_i) = \dot{\varphi}_1(\mathbf{\bar{x}}_i) = \mathbf{v}_i$.
From the definition of $\mathcal{A}$ it then follows that we have proved the
following fact about the vertical part of $\dot{\varphi}_1$:
\begin{align}
  (\mathcal{A}_{(\varphi, g)}(\dot{\varphi}_1, \dot{g}_1))(\mathbf{\bar{x}}_i) 
    & = (\varphi^\ast \mathbf{u}_{\mathrm{v}})(\bar{\mathbf{x}}_i) \nonumber \\
    & = T \varphi^{-1} (\mathbf{v}_i - \nabla \Phi_1 (\mathbf{x}_i)) \label{vf}
\end{align}
for  $i = 1, \ldots, N$.
A similar property holds for $\dot{\varphi}_2$, with $\mathbf{v}_i$
replaced by $\mathbf{w}_i$ and involving $\Phi_2$, the solution of
(\ref{neumann}) with boundary data $\dot{g}_2$.

As an aside, we note that the Neumann connection induces a connection
on the trivial bundle $\mathbb{R}^2 \times \operatorname{SE}(2)  \rightarrow \operatorname{SE}(2) $, and the
expression between brackets on the right-hand side of (\ref{vf}) is
the vertical component of the vector $\mathbf{v}_i$.  This connection is
important in the theory of \emph{Routh reduction} (see
\cite{Routh00}).

\paragraph{The Curvature Term.}

The curvature term $\beta_{{\mathrm{curv}}}$ is given by
\begin{multline*}
  \beta_{\mathrm{curv}}(\mathbf{x}_1, \ldots, \mathbf{x}_N; g)( (\mathbf{v}_1, \ldots,
  \mathbf{v}_N, \dot{g}_1), (\mathbf{w}_1, \ldots, \mathbf{w}_N, \dot{g}_2))
  \\
  = \left< \mu, \mathcal{B}(\varphi, g)( (\dot{\varphi}_1, \dot{g}_1),
    (\dot{\varphi}_2, \dot{g}_2) )\right>,
\end{multline*}
where $\mathcal{B}$ is the curvature of the Neumann connection.  It follows
from (\ref{curvature}) that $\beta_{\mathrm{curv}}$ does not depend on the
value of $\mathbf{v}_i$ and $\mathbf{w}_i$, but only on $\dot{g}_1$ and
$\dot{g_2}$.  From now on, we will therefore suppress these arguments.

\begin{proposition}
  For the rigid body interacting with $N$ point vortices, the
  curvature term is  
given by 
\begin{equation} \label{betacurv} \beta_{\mathrm{curv}}(\mathbf{x}_1, \ldots,
  \mathbf{x}_N; g)(\dot{g}_1, \dot{g}_2) = \sum_{i=1}^N \Gamma_i \, d \textbf{x} \,
  (\nabla \Phi_1(\mathbf{x}_i), \nabla \Phi_2(\mathbf{x}_i)),
  \end{equation}
  where $\Phi_1$ and $\Phi_2$ are the solutions of (\ref{neumann})
  associated to $\dot{g}_1$ and $\dot{g}_2$, respectively, and $d \textbf{x}$ is the volume element on $\mathcal{F}$.
\end{proposition}
\begin{proof}
  We use the explicit form for the elementary velocity potentials
  derived in proposition~\ref{prop:invariance}.  Note that
  $\beta_{\mathrm{curv}}$ depends on $\dot{g}_1$ and $\dot{g}_2$ through
  $g^{-1}\dot{g}_1$ and $g^{-1}\dot{g}_2$, so we can calculate its
  expression on the basis 
$\{\mathbf{e} _\Omega, \mathbf{e} _x, \mathbf{e}_y\}$ of 
$\mathfrak{se}(2)$ given in (\ref{sebasis})
  by assuming that $g = e$ and
  $\dot{g}_1, \dot{g}_2 \in T_e \operatorname{SE}(2)  = \mathfrak{se}(2)$.

  First, we prove that the boundary term in (\ref{curvature}) is zero.
  This is obvious if either $\dot{g}_1$ or $\dot{g}_2$ is proportional
  to $\mathbf{e}_\Omega$, since $\Phi_\omega = 0$.  We may therefore
  assume that $\dot{g}_1 = \mathbf{e}_x$ and $\dot{g}_2 = \mathbf{e}_y$.  In
  this case, the two-form $\mathbf{d} \Phi_x \wedge \mathbf{d} \Phi_y$ restricted to
  the boundary of $\mathcal{F}$ (the curve $(x - x_0)^2 + (y -
  y_0)^2 = 1$) is given by $\mathbf{d} \Phi_x \wedge \mathbf{d} \Phi_y = d \textbf{x}$.
  Since $\ast d \textbf{x}= 1$ in two dimensions, the boundary term reduces
  to
\begin{equation} \label{curvcirc}
\int_{\partial \mathcal{F}} \alpha = \Gamma,
\end{equation}
which is zero, since by assumption there is no circulation around the
rigid body.

This leaves us with the integral over the fluid domain in
(\ref{curvature}).  Plugging in the explicit form (\ref{vorticity}) of
the vorticity, we have
\begin{align*}
\beta_{\mathrm{curv}}(\mathbf{x}_1, \ldots,
  \mathbf{x}_N; g)(\dot{g}_1, \dot{g}_2) & = 
\left\langle\!\left\langle \mu, \mathbf{d} \Phi_1 \wedge \mathbf{d} \Phi_2\right\rangle\!\right\rangle_\mathcal{F} \\ 
& = \sum_{i = 1}^N \Gamma_i
\,d \textbf{x}\,(\nabla \Phi_1(\mathbf{x}_i), \nabla \Phi_2(\mathbf{x}_i)).
\end{align*}
We do not work out this expression any further, since it will be
cancelled out by a similar term in the computation of $\beta_{{\mathrm{l.p}}}$.
\end{proof}

It is remarkable that the contribution from the curvature to the
magnetic term will be cancelled in its entirety by a similar term in
the expression for Lie-Poisson term, which is computed below.
However, in the case of non-zero circulation, the curvature generates
an additional term (\ref{curvcirc}) proportional to the circulation.
The effect of this term is studied in \cite{VaKaMa08}.

\paragraph{The Lie-Poisson Term.}

The Lie-Poisson term $\beta_{{\mathrm{l.p}}}$ is given by
\begin{multline*}
  \beta_{\mathrm{l.p}}(\mathbf{x}_1, \ldots, \mathbf{x}_N; g)( (\mathbf{v}_1, \ldots,
  \mathbf{v}_N, \dot{g}_1), (\mathbf{w}_1, \ldots, \mathbf{w}_N, \dot{g}_2))
  \\
  = \left< \mu, [\mathcal{A}_{(\varphi, g)}(\dot{\varphi}_1, \dot{g}_1), 
\mathcal{A}_{(\varphi, g)}(\dot{\varphi}_2, \dot{g}_2)] \right>,
\end{multline*}
where $(\dot{\varphi}_1, \dot{g}_1)$ and $(\dot{\varphi}_2,
\dot{g}_2)$ have similar interpretations as before.  The right-hand
side can be made more explicit by noting that, for divergence-free
vector fields $\mathbf{u}_1$ and $\mathbf{u}_2$ tangent to $\partial \mathcal{F}$
and arbitrary one-forms $\alpha$, the following holds: 
\begin{equation} \label{magic}
  \int_\mathcal{F} \alpha([\mathbf{u}_1, \mathbf{u}_2]) \,d \textbf{x} = 
  \int_\mathcal{F} \mathbf{d} \alpha(\mathbf{u}_1, \mathbf{u}_2) \,d \textbf{x}.
\end{equation}
The proof of this assertion is a straightforward application of
Cartan's magic formula and parallels the proof of 
theorem~4.2 in \cite{mw1983}.

Using this formula, we have 
\begin{multline*}
  \left< \mu, [\mathcal{A}_{(\varphi, g)}(\dot{\varphi}_1, \dot{g}_1),
     \mathcal{A}_{(\varphi, g)}(\dot{\varphi}_2, \dot{g}_2)] \right> =
  \int_\mathcal{F} (\varphi^{\ast}  \alpha)([\mathcal{A}_{(\varphi,
    g)}(\dot{\varphi}_1, \dot{g}_1),
   \mathcal{A}_{(\varphi, g)}(\dot{\varphi}_2, \dot{g}_2)]) \, d \textbf{x} \\
    = \int_\mathcal{F} \mu(\mathcal{A}_{(\varphi,
    g)}(\dot{\varphi}_1, \dot{g}_1),
  \mathcal{A}_{(\varphi, g)}(\dot{\varphi}_2, \dot{g}_2)) \, d \textbf{x} \\
    = \sum_{i=1}^N \Gamma_i \, (\varphi^{\ast} d \textbf{x})\big((\mathcal{A}_{(\varphi,
    g)}(\dot{\varphi}_1, \dot{g}_1))(\mathbf{\bar{x}}_i),
  (\mathcal{A}_{(\varphi, g)}(\dot{\varphi}_2, \dot{g}_2))(\mathbf{\bar{x}}_i)\big).
\end{multline*}
By substituting (\ref{vf}) in this expression, we conclude that the Lie-Poisson term is given by 
\begin{multline*}
  \beta_{\mathrm{l.p}}(\mathbf{x}_1, \ldots, \mathbf{x}_N; g)( (\mathbf{v}_1, \ldots,
  \mathbf{v}_N, \dot{g}_1), (\mathbf{w}_1, \ldots, \mathbf{w}_N, \dot{g}_2)) \\
= \sum_{i=1}^N \Gamma_i \, d \textbf{x}(
\mathbf{v}_i - \nabla \Phi_1 (\mathbf{x}_i), \mathbf{w}_i - \nabla \Phi_2 (\mathbf{x}_i)).
\end{multline*}

\paragraph{The Magnetic Two-form: Putting Everything Together.}

Using the previously derived results, we may conclude that $\beta_\mu$
is given by (suppressing its arguments for the sake of clarity)
\begin{align*}
  \beta_\mu & = \beta_{\mathrm{curv}} - \beta_{\mathrm{l.p}} \\
  & = \sum_{i=1}^N \Gamma_i \big( d \textbf{x}(\nabla \Phi_1(\mathbf{x}_i),
    \nabla \Phi_2(\mathbf{x}_i)) - d \textbf{x}(\mathbf{v}_i - \nabla
    \Phi_1(\mathbf{x}_i), \mathbf{w}_i - \nabla \Phi_2(\mathbf{x}_i)\big) \\
  & = \sum_{i=1}^N \Gamma_i \big( 
    -d \textbf{x}(\mathbf{v}_i, \mathbf{w}_i) + d \textbf{x}(\mathbf{v}_i, \nabla
    \Phi_2(\mathbf{x}_i)) + d \textbf{x}(\nabla \Phi_1(\mathbf{x}_i), \mathbf{w}_i) \big).
\end{align*}

Introducing stream functions $\Psi'_i(g, \dot{g}; \mathbf{x})$ ($i =
1,2$) as harmonic conjugates to $\Phi_i(g, \dot{g}; \mathbf{x})$,
\emph{i.e.} such that
\begin{equation} \label{HC}
{\nabla \Phi_i} \contraction d \textbf{x} = \mathbf{d} \Psi_i' \quad i = 1, 2
\end{equation}
we may rewrite the expression for $\beta_\mu$ as 
\[
\beta_\mu = \sum_{i = 1}^N \Gamma_i \big(-d \textbf{x}(\mathbf{v}_i, \mathbf{w}_i)
- \mathbf{d} \Psi'_2(\mathbf{v}_i) + \mathbf{d}\Psi'_1(\mathbf{w}_i) \big).
\]

The stream function can be written in Lamb form by introducing the
elementary stream functions $\Psi_x$, $\Psi_y$, and $\Psi_\omega$ as
harmonic conjugates to $\Phi_x$, $\Phi_y$, and $\Phi_\omega$,
respectively.  For $i = 1, 2$, we then have 
\begin{align*}
  \Psi_i'(g, \dot{g}; \mathbf{x}) & = v_{x, i} \Psi_x(g, \mathbf{x}) 
  + v_{y, i} \Psi_y(g, \mathbf{x}) + v_{\omega, i} \Psi_\omega(g,
  \mathbf{x}) \\
  & = \mathbf{d} x_0 (\dot{g}_i) \Psi_x(g, \mathbf{x}) +
      \mathbf{d} y_0 (\dot{g}_i) \Psi_y(g, \mathbf{x}) +
      \mathbf{d} \beta (\dot{g}_i) \Psi_\omega(g, \mathbf{x}), 
\end{align*}
and so the magnetic two-form becomes
\[
\beta_\mu = \sum_{i = 1}^N \Gamma_i \big(-d \textbf{x}(\mathbf{v}_i, \mathbf{w}_i) -
\mathbf{d} \xi_A(\dot{g}_2) \mathbf{d} \Psi'_A(\mathbf{v}_i) + \mathbf{d} \xi_A(\dot{g}_1)
\mathbf{d}\Psi'_1(\mathbf{w}_i) \big), 
\]
where the coordinates on $\operatorname{SE}(2) $ are denoted by $\xi_A = \{x_0,
y_0, \beta\}$, and $\Psi'_A = \{\Psi'_x, \Psi'_y,
\Psi'_\omega\}$.  

\begin{theorem} \label{thm:magform}
  The magnetic two-form $\beta_\mu$ is a two-form on $\mathbb{R}^{2N} \times
  \operatorname{SE}(2) $ and can be written as
\begin{multline} \label{betaexpl}
  \beta_\mu(\mathbf{x}_1, \ldots, \mathbf{x}_N; g)( (\mathbf{v}_1, \ldots,
  \mathbf{v}_N, \dot{g}_1), (\mathbf{w}_1, \ldots, \mathbf{w}_N, \dot{g}_2)) \\
  = \sum_{i=1}^N \Gamma_i\, \big( -d \textbf{x}(\mathbf{v}_i, \mathbf{w}_i) + \mathbf{d}
 \Theta(g, \mathbf{x}_i)
(\dot{g}_1, \mathbf{v}_i; \dot{g}_2,
  \mathbf{w}_i) \big).
\end{multline}
where $\Theta$ is the one-form on $\mathbb{R}^2 \times \operatorname{SE}(2) $ given by 
\begin{equation} \label{theta}
  \Theta(g, \mathbf{x})(\dot{g}, \mathbf{v}) = \Psi_A'(g, \mathbf{x})
  \mathbf{d}\xi_A(\dot{g}).   
\end{equation}

\end{theorem}
\begin{proof}
  Note that the exterior derivative on the right-hand side of
  (\ref{betaexpl}) is the exterior derivative on $\mathbb{R}^{2N} \times
  \operatorname{SE}(2) $.  Expanding the second term, we have
\begin{align*}
  \mathbf{d}(\Psi_A'(g, \mathbf{x}) \mathbf{d}\xi_A) & = \mathbf{d} \Psi_A' \wedge \mathbf{d} \xi_A \\
    & = \mathbf{d} \Psi'_x \wedge \mathbf{d} x_0 + \mathbf{d} \Psi'_y \wedge \mathbf{d} y_0 \\
    & = \mathbf{d}_{\mathbb{R}^2} \Psi'_x \wedge \mathbf{d} x_0 + \mathbf{d}_{\mathbb{R}^2} \Psi'_y \wedge \mathbf{d}
    y_0 + \left( \frac{\partial \Psi'_y}{\partial x_0} -  
      \frac{\partial \Psi'_x}{\partial y_0} \right) \mathbf{d} x_0
    \wedge \mathbf{d} y_0,
\end{align*}
where the subscript `$\mathbb{R}^2$' indicates that the derivative should be
taken with respect to the spatial variables only, \emph{i.e.} that the
$\operatorname{SE}(2) $-variable should be kept fixed.  Also, we rely on the fact that
$\Psi'_\omega = 0$ (since $\Phi_\omega = 0$), and that $\Psi'_{x, y}$
does not depend on $\beta$.

The theorem is proved once we show that the term in brackets
vanishes.  To show this, we need the explicit form for the elementary
stream functions: 
\begin{equation} \label{stream}
  \Psi'_x = \frac{y-y_0}{(x-x_0)^2 + (y-y_0)^2} \quad \text{and} \quad  
\Psi'_y = -\frac{x-x_0}{(x-x_0)^2 + (y-y_0)^2}.
\end{equation}
and it is then easily shown that this is indeed the case.
\end{proof}

Putting everything together, we conclude that the symplectic structure
$\Omega_\mathcal{B}$ on $\mathbb{R}^{2N} \times T^{\ast}  \operatorname{SE}(2) $ (see
(\ref{magsymp})) is given by 
\begin{equation} \label{magsympexpl}
  \Omega_\mathcal{B} = \Omega_{\mathrm{can}} - \beta_\mu, 
\end{equation}
where $\beta_\mu$ is given in the theorem above.  The first term is the canonical symplectic form on 
$T^{\ast}  \operatorname{SE}(2)$ while the magnetic consists of two parts, one (proportional to $d \mathbf{x}$) being the symplectic form on $\mathbb{R}^{2N}$, while the other one encodes the effects of the ambient fluid through the stream functions $\Psi_x'$ and $\Psi_y'$.  Note that $\mathbb{R}^{2N}$ is a co-adjoint orbit of the diffeomorphism group, and that its symplectic structure is nothing but the Konstant-Kirillov-Souriau form.

\section{Reduction with Respect to $SE(2)$} \label{sec:redse}

After reducing the fluid-solid system with respect to the $\mathrm{Diff}_{\mathrm{vol}}$-symmetry, we are left with a system on $T^{\ast}  \operatorname{SE}(2)  \times \mathbb{R}^{2N}$ with a magnetic symplectic form.  The group $\operatorname{SE}(2) $ acts diagonally on the reduced phase space. What remains is to carry out the reduction with respect to the group $\operatorname{SE}(2)$ and then to make the connection with the equations of motion in \cite{cylvortices} and \cite{BoMa03}.  

Two factors complicate this strategy, however: first of all, there is the presence of the magnetic term in the symplectic structure, which precludes using standard Euler-Poincar\'e theory, for instance.  Secondly, the phase space is a Cartesian product on which the symmetry group acts diagonally.  Even in the absence of magnetic terms, this would lead to the appearance of certain \textbf{\emph{interaction terms}} in the Poisson structure (see
\cite{KrMa86}).

In order to do Poisson reduction for this kind of manifold, 
we begin this section by considering the product of a cotangent bundle $T^{\ast}  H$, where $H$ is a Lie group, and a symplectic manifold $P$, and we
assume that the symplectic structure on $T^{\ast}  H \times P$ is the sum
of the canonical symplectic structures on both factors plus an
additional magnetic term.  This case extends both magnetic Lie-Poisson
reduction in \cite{MarsdenHamRed}, for which the $P$-factor is absent,
as well as the ``coupling to a Lie group'' scenario of \cite{KrMa86},
for which there is no magnetic term.   Since the proofs in this section are rather lengthy, we have relegated them to Appendix~\ref{appB} and we simply quote the relevant expressions here.

%

The bulk of this section is then devoted to making these results explicit
for the case where $H = \operatorname{SE}(2) $, $P = \mathbb{R}^{2N}$ and the symplectic
form is the magnetic symplectic form derived in the previous section; see equations \eqref{betaexpl} and \eqref{magsympexpl}.
The induced Poisson structure with interaction terms will turn out to
be nothing but the BMR Poisson structure (\ref{bmrbracket}).
Secondly, the Poisson map induced by the momentum map $J$ is just the
Shashikanth map described in Theorem~\ref{thm:shashi}, so that the
transformed Poisson structure is the SMBK Poisson structure
(\ref{PoissSMBK}).

\subsection{Coupling to a Lie Group} \label{sec:gen}

As mentioned above, the purpose of this section is to generalize the reduction of phase spaces of the form $T ^{\ast} H \times P $, where $H$ is a Lie group acting on the symplectic manifold $P$. This is more general than results in the literature, since the symplectic structure on $T ^{\ast} H \times P $ may contain magnetic terms. We shall need this generalization in our application.  The proofs of the theorems in this section can be found in appendix~\ref{appB}.

Specifically, if $\omega _P$ denotes the symplectic form on $P$, we assume that
the symplectic structure $\Omega$ on $T^{\ast}  H \times P$ is given by 
\begin{equation} \label{gensympform}
  \Omega = \omega_{{\mathrm{can}}} - B,
\end{equation}
where $\omega_{\mathrm{can}}$ is the canonical symplectic structure on $T^{\ast} 
H$ and $B$ is a closed two-form on $G \times P$ such that $B_{|P} =
- \omega_P$, \emph{i.e.} the restriction of $B$ to $P$ is just
$-\omega_P$.  The latter ensures that $\Omega$ is a well-defined
symplectic form on $T^{\ast}  H \times P$. 

We assume furthermore that $H$ acts on $P$ and we denote the action of
$h \in H$ on $p \in P$ by $h \cdot p$.  Furthermore, we assume that
this action leaves invariant $\omega_P$.  The group $H$ then acts
diagonally on $T^{\ast}  H \times P$: if $h \in H$ and $(\alpha_g, p) \in
T^{\ast}  H \times P$, then 
\begin{equation} \label{diagaction}
  h \cdot (\alpha_g, p) = (T^{\ast}  L_{h^{-1}}(\alpha_g), h\cdot p).
\end{equation}
%

\paragraph{The Momentum Map.}

Consider an element $\xi$ of $\mathfrak{h} $.  We denote the fundamental
vector field associated to $\xi$ on $T^{\ast}  H \times P$ by 
\[
  \xi_{T^{\ast}  H \times P}(\alpha_g, p) := 
  \frac{d}{dt} \exp(t\xi) \cdot (\alpha_g, p) \Big|_{t = 0},
\]
and similar for the fundamental vector field $\xi_{H \times P}$ on $H
\times P$.
Under certain topological assumptions (which are satisfied for the
solid-fluid system), there exists a momentum map $J: T^{\ast}  H \times P
\rightarrow \mathfrak{h} ^{\ast} $ for the action of $H$, defined by 
\[
  \mathbf{i}_{\xi_{T^{\ast}  H \times P}} \Omega = \mathbf{d} J_\xi, 
\]
where $J_\xi = \left<J, \xi\right>$.  It follows from the explicit form for
the symplectic form (\ref{gensympform}) that $J$ can be written as $J
= J_{\mathrm{can}} - \phi$, where $J_{\mathrm{can}}: T^{\ast}  H \rightarrow \mathfrak{h} ^{\ast} $ is
the momentum map associated to the canonical symplectic form on
$T^{\ast}  H$:
\begin{equation} \label{mommap}
  \mathbf{i}_{\xi_{T^{\ast}  H}} \omega_{\mathrm{can}} = \mathbf{d} \left<J_{\mathrm{can}}, \xi\right>
\end{equation}
while $\phi: H \times P \rightarrow \mathfrak{h} ^{\ast} $ is the
so-called \emph{$\mathcal{B}\mathfrak{g}$-potential} (borrowing the terminology of
\cite{MarsdenHamRed}):
\begin{equation} \label{bgpot}
  \mathbf{i}_{\xi_{H \times P}} B = \mathbf{d} \left<\phi, \xi\right>.
\end{equation}

The theory of $\mathcal{B}\mathfrak{g}$-potentials and Poisson reduction for
cotangent bundles $T^{\ast}  H$ endowed with a magnetic symplectic form
is further developed in \cite{MarsdenHamRed}.  Here, we are dealing
with a product $T^{\ast}  H \times P$, for which no such theory can be
found in the current literature.

\paragraph{Infinitesimal Equivariance of the Momentum Map.} Whereas
the canonical momentum map $J_{\mathrm{can}}$ is equivariant, the same does not
necessarily hold for the $\mathcal{B}\mathfrak{g}$-potential.  To measure
non-equivariance, we introduce a one-cocycle $\sigma: H \rightarrow
L(\mathfrak{h} , \mathcal{C}(H\times P))$ by defining first a family of functions
$\Gamma_{\eta, g}$ (where $\eta \in \mathfrak{h} $ and $g \in H$):
\[
\Gamma_{\eta, g}(h, p) = - \left<\phi(gh, gp), \eta\right> + \left<\mathrm{Ad}_g^{\ast} 
  \phi(h, p), \eta\right>, 
\]
for all $(h, p) \in H \times P$, 
and then putting $\sigma(g)\cdot \eta = \Gamma_{\eta, g}$.  This
definition follows the usual introduction of cocycles for momentum
maps; see for instance \cite{GuSt1984, MarsdenRatiu}.

The one-cocycle $\sigma$ induces a two-cocycle $\Sigma : \mathfrak{h}  \times
\mathfrak{h}  \rightarrow C^\infty(H \times P)$ given by 
\[
   \Sigma(\xi, \eta) = T_e \sigma_\eta(\xi), 
\]
where $\sigma_\eta : H \rightarrow C^\infty(H \times P)$ is defined by
$\sigma_\eta(g) = \sigma(g)\cdot \eta$.  It is not hard to verify that
$\Sigma$ is explicitly given by 
\begin{equation} \label{twococycle}
   \Sigma(\xi, \eta) = -\left<\phi, [\xi, \eta]\right> + B(\xi_{H\times P},
   \eta_{H\times P}).
\end{equation}

\paragraph{The Poisson Structure on $\mathfrak{h} ^{\ast}  \times P$.}

The canonical Poisson structure on $T^\ast H \times P$ associated to the symplectic structure (\ref{gensympform}) gives rise to a Poisson structure on the quotient space $\mathfrak{h}^\ast \times P$, which we denote by 
$\{\cdot, \cdot\}_{\mathrm{int}}$.  

The explicit form of this Poisson structure is derived in Appendix~\ref{appB}.  Before quoting this result, we 
first
introduce an operation $\star: \mathfrak{h}  \times \mathfrak{h} 
\rightarrow C^\infty(H \times P)$ defined by
\[
  \xi \star \eta := \left\{ \phi_{\xi|P},\phi_{\eta|P}\right\}_P,  
\]
where $\{\cdot, \cdot\}_P$ is the Poisson structure associated to
$\omega_P$, and an operation $\lhd: \mathfrak{h}  \times C^\infty(P)
\rightarrow C^\infty(H \times P)$ by putting
\[
  \xi \lhd F := \{\phi_{\xi|P}, F\}_P.
\]

Using these two operations, the reduced Poisson structure is given in the following theorem.  Note the interaction term due to curvature and the last two terms which are due to the coupling of the Lie group $H$ with the symplectic manifold $P$.

\begin{theorem} \label{thm:bmr}
  The reduced Poisson structure on $\mathfrak{h} ^{\ast}  \times P$ is given by 
\begin{align} 
  \{f, k\}_{\mathrm{int}} & = {\frac{\delta f}{\delta \mu}} \star {\frac{\delta k}{\delta \mu}} - \{f_{|P}, k_{|P}\}_P -
  B\left(\left({\frac{\delta f}{\delta \mu}}\right)_{H \times P}, \left({\frac{\delta k}{\delta \mu}}\right)_{H \times
      P}\right)  \nonumber \\
  & - {\frac{\delta f}{\delta \mu}} \lhd k + {\frac{\delta k}{\delta \mu}} \lhd f, \label{redpoisson}
\end{align}
for functions $f = f(\mu, x)$ and $k = k(\mu, x)$ on $\mathfrak{h} ^{\ast} 
\times P$.
\end{theorem}

\paragraph{Shifting Away the Interaction Terms.} 
The reduced Poisson structure $\{\cdot, \cdot\}_{\mathrm{int}}$ described in theorem~\ref{thm:bmr} is not canonical, \emph{i.e.} it is not the sum of the Lie-Poisson structure on $\mathfrak{h}^\ast$ and the canonical Poisson structure on $\mathbb{R}^{2N}$, but rather contains a number of 
\textbf{\emph{interaction terms}}.  However, using the $\Bg$-potential $\phi$, we can define a \textbf{\emph{shift map}} $\hat{\mathcal{S}}$ from $\mathfrak{h}^\ast \times \mathbb{R}^{2N}$ to itself, taking $\{\cdot, \cdot\}_{\mathrm{int}}$ into the canonical Poisson structure.
The price we have to pay for getting rid of interaction terms
is the introduction of the non-equivariance cocycle $\Sigma$ of $\phi$, as in the following definition.

\begin{definition} \label{def:natpoisson}
The natural Poisson structure on $\mathfrak{h} ^{\ast}  \times P$
is given by 
  \begin{equation} \label{poissonsimple} \{f, k\}_\Sigma =
    \{f_{|\mathfrak{h} ^{\ast} }, k_{|\mathfrak{h} ^{\ast} }\}_{\mathrm{l.p}} - \{f_{|P},
    k_{|P}\}_P - \Sigma\left( \frac{\delta f}{\delta\mu}, \frac{\delta
        k}{\delta\mu} \right)
  \end{equation}
  for functions $f = f(\mu, x)$ and $k = k(\mu, x)$ on $\mathfrak{h} ^{\ast} 
  \times P$.
\end{definition}

The main result of this section is then given in the following theorem.  It related the Poisson structure with interaction terms with the natural one (possibly with a cocycle).  The former will turn out the BMR Poisson structure, while the latter is nothing but the SMBK structure.  The cocycle will encode the effects of nonzero circulation on the rigid body.

\begin{theorem} \label{thm:shift}
  The map $\hat{\mathcal{S} }: \mathfrak{h} ^{\ast}  \times P \rightarrow \mathfrak{h} ^{\ast} 
  \times P$ given by 
\begin{equation} \label{hatshift}
  \hat{\mathcal{S} }(\mu, x) = (\mu - \phi(e, x), x).
\end{equation}  
  is a Poisson isomorphism taking the Poisson structure
  $\{\cdot, \cdot\}_{\mathrm{int}}$ with interaction terms into the Poisson
  structure $\{\cdot, \cdot\}_\Sigma$:
  \begin{equation} \label{poissonmap}
    \{ f \circ \hat{\mathcal{S} }, k \circ \hat{\mathcal{S} } \}_{\mathrm{int}} = \{f, k\}_\Sigma \circ
    \hat{\mathcal{S} }.   
  \end{equation}
\end{theorem}

\paragraph{The Symplectic Leaves.}  The shifted Poisson
structure $\{\cdot, \cdot\}_\Sigma$ is the sum of the Lie-Poisson
structure on $\mathfrak{h} ^{\ast} $, the canonical Poisson structure on $P$,
and a cocycle term.  If $\Sigma = 0$, this allows us to write down a
convenient expression for the symplectic leaves in $\mathfrak{h} ^{\ast}  \times
P$: 

\begin{proposition} \label{prop:sympleaf} For $\Sigma = 0$, the symplectic
  leaves in $\mathfrak{h} ^{\ast}  \times P$ of the Poisson structure $\{\cdot,
  \cdot\}_{\Sigma = 0}$ are of the form $\mathcal{O}_\mu \times P$, where
  $\mathcal{O}_\mu$ is the co-adjoint orbit of an element $\mu \in
  \mathfrak{h} ^{\ast} $.
\end{proposition}

The proof follows that of proposition~10.3.3 in \cite{MarsdenHamRed}
and relies on the fact that the symplectic leaves are precisely the
symplectic reduced spaces.  But since the Poisson structure on the
reduced space is simply the sum of the Lie-Poisson and the canonical
Poisson structure, the reduced space at $\mu$ is $\mathcal{O}_\mu \times P$,
as above.  When $\Sigma \ne 0$, we expect the co-adjoint orbit
$\mathcal{O}_\mu$ to be replaced by an orbit $\mathcal{O}_\mu^\Sigma$ of a
suitable \textbf{\emph{affine}} action of $H$ on $\mathfrak{h} ^{\ast} $.

\subsection{The Fluid-Solid System}

Now we are ready to specialize the theory in the previous sections to the case of a solid interacting dynamically with point vortices. Recall that our goal is to find the Hamiltonian structure for this problem using reduction techniques and to use the shifting map developed in the preceding section to relate the two Hamiltonian structures in the literature.

Recall that the reduced phase space for the solid-fluid system is
$T^{\ast}  \operatorname{SE}(2)  \times \mathbb{R}^{2N}$.  Now, the group $\operatorname{SE}(2) $ acts
on $\operatorname{SE}(2)  \times \mathbb{R}^{2N}$ by the diagonal left action, denoted by
$\Phi$ and given in inertial coordinates by
\[
\Phi_h(g; \mathbf{x}_1, \ldots, \mathbf{x}_N) = (hg; h\mathbf{x}_1, \ldots,
h\mathbf{x}_N).
\]
Hence, $\operatorname{SE}(2) $ acts from the left on $T^{\ast}  \operatorname{SE}(2)  \times \mathbb{R}^{2N}$
using the cotangent lift in the first factor,  and thus, this action leaves the Hamiltonian
(\ref{totalhamiltonian}) invariant.

This system is of the form $T^{\ast}  H \times P$, as discussed in the
previous section, where $H = \operatorname{SE}(2) $ and $P = \mathbb{R}^{2N}$.  We now apply
the results of that section to divide out the $\operatorname{SE}(2) $-symmetry and
obtain a system on $\mathfrak{se}(2)^{\ast}  \times \mathbb{R}^{2N}$.  

This reduction is similar to the passage from inertial to body coordinates for the
rigid body (see for example \cite{MarsdenRatiu}).
To see this, notice that the space $\mathfrak{se}(2)^{\ast}  \times \mathbb{R}^{2N}$
is obtained from $T^{\ast}  \operatorname{SE}(2)  \times \mathbb{R}^{2N}$ by dividing out by the
diagonal $\operatorname{SE}(2) $-action: the quotient mapping is given by
\[
  \rho : T^{\ast}  \operatorname{SE}(2)  \times \mathbb{R}^{2N}  \rightarrow 
  \mathfrak{se}(2)^{\ast}  \times \mathbb{R}^{2N},
\]
defined as 
\[
  \rho(g, \alpha_g; \mathbf{x}_1, \ldots, \mathbf{x}_N) = 
  (T^{\ast}  L_g(\alpha_g); \mathbf{X}_1, \ldots, \mathbf{X}_N), 
\]
where, if $g = (R, \mathbf{x}_0)$, then $\mathbf{x}_i$ and $\mathbf{X}_i$ are
related by 
\begin{equation} \label{transition}
  \mathbf{x}_i = R\mathbf{X}_i + \mathbf{x}_0.  
\end{equation}
In other words, if $\mathbf{x}_i$ describes the location of the $i$th
vortex in inertial coordinates, then $\mathbf{X}_i$ is its location in
a frame fixed to the body.

\begin{proposition} The magnetic symplectic structure
 \textup{ (\ref{magsympexpl})} is invariant under the action of $\operatorname{SE}(2) $ on
  $T^{\ast}  \operatorname{SE}(2)  \times \mathbb{R}^{2N}$ described above.
\end{proposition}
\begin{proof}
  Recall the expression (\ref{magsympexpl}) for $\Omega_\mathcal{B}$, where
  $\beta_\mu$ is given by (\ref{betaexpl}).  It is a standard result
  (see for instance \cite{CuBa97}) that the canonical symplectic form
  $\Omega_{\mathrm{can}}$ on $T^{\ast}  \operatorname{SE}(2) $ is invariant under the left action
  of $\operatorname{SE}(2) $, and a similar result holds for the form $\sum_{i = 1}^N
  \Gamma_i \mu(\mathbf{v}_i, \mathbf{w}_i)$.  The only thing that remains to
  be shown is that $\Theta$ is $\operatorname{SE}(2) $-invariant, but this follows
  from the $\operatorname{SE}(2) $-invariance of $\Psi$, which is itself a consequence
  of proposition~\ref{prop:invariance} and the fact that $\Phi$ and
  $\Psi$ are harmonic conjugates.
 \end{proof}

\paragraph{The Momentum Map.}

Recall from section~\ref{sec:gen} that the momentum map $J$ for the
action of $\operatorname{SE}(2) $ on $T^{\ast}  \operatorname{SE}(2)  \times \mathbb{R}^{2N}$ is the difference
of two separate parts: $J = J_{\mathrm{can}} - \phi$, where $J_{\mathrm{can}}$ is the
momentum map (\ref{mommap}) defined by the canonical symplectic form
on $T^{\ast}  \operatorname{SE}(2) $, while $\phi$ is the so-called $\Bg$-potential
(\ref{bgpot}).

The momentum map $J$ is a map from $T^{\ast}  \operatorname{SE}(2)  \times \mathbb{R}^{2N}$ to
$\mathfrak{se}(2)^{\ast} $.  It will be convenient to identify $T^{\ast}  \operatorname{SE}(2) $
with $\operatorname{SE}(2)  \times \mathfrak{se}(2)^{\ast} $ through left translations, and to
use the fact that $\mathfrak{se}(2)^{\ast} $ is isomorphic to $\mathbb{R}^3$, so that
$J$ is a collection of three functions $(J^x, J^y, J^\Omega)$ on
$\mathbb{R}^{2N} \times \operatorname{SE}(2)  \times \mathfrak{se}(2)^{\ast} $.  A typical element of
that space will be denoted as $(\mathbf{x}_1, \ldots, \mathbf{x}_N; g,
\mathbf{\Pi})$, where $g = (R, \mathbf{x}_0)$ and $\mathbf{\Pi} =
(\Pi_x, \Pi_y, \Pi_\omega)$, but for the sake of clarity we will
usually suppress the argument of $J$.

\begin{proposition}
  The momentum map $J$ associated to the $\operatorname{SE}(2) $-symmmetry represents
  the spatial momentum of the solid-fluid system, and is given by $J =
  (J_x, J_y, J_\Omega)$, where  
  \begin{align}
\begin{pmatrix}
J_x \\ J_y 
\end{pmatrix}
& = R 
\begin{pmatrix}
  \Pi_x \\ \Pi_y 
\end{pmatrix}
+ \sum_{i=1}^N \Gamma_i \mathbf{x}_i \times \mathbf{e}_3 - \sum_{i=1}^N
\Gamma_i 
\begin{pmatrix}
  \Psi_x(\mathbf{x}_i) \\ \Psi_y(\mathbf{x}_i)
\end{pmatrix}
+ \Gamma \mathbf{x}_0 \times \mathbf{e}_3 \label{fullmommmap} \\
J_\Omega & = \Pi_\Omega - \sum_{i=1}^N \frac{\Gamma_i}{2} (x_i^2 +
y_i^2) - \sum_{i=1}^N\Gamma_i \Psi_\omega(\mathbf{x}_i), \nonumber
  \end{align}
  where $\Gamma = \sum_{i = 1}^N \Gamma_i$ is the total vortex
  strength.
\end{proposition}

\begin{proof}
The canonical part $J_{\mathrm{can}}$ can be obtained through a standard 
calculation for cotangent lifted actions (see \cite{MarsdenRatiu}).  The result is
\[
\begin{pmatrix}
J_x \\ J_y 
\end{pmatrix}
= R 
\begin{pmatrix}
  \Pi_x \\ \Pi_y 
\end{pmatrix}
\quad \text{and} \quad  J_\Omega
= \Pi_\Omega,
\]
for $g = (R, \mathbf{x}_0)$ and $\mathbf{\Pi} = (\Pi_x, \Pi_y,
\Pi_\Omega)$.  

For the $\Bg$-potential, note that the form $\beta_\mu$ is the sum of
a ``pure vortex'' part (\emph{i.e.} not involving the fluid) and a
part involving the stream functions (\emph{i.e.} the form $\Theta$).

The $\Bg$-potential corresponding to the pure-vortex part is given by
\[
  \begin{pmatrix}
    \phi_{\mathrm{vortex}}^x \\
    \phi_{\mathrm{vortex}}^y 
  \end{pmatrix} = - \sum_{i=1}^N \mathbf{x}_i \times \mathbf{e}_3
\quad \text{and} \quad 
  \phi_{\mathrm{vortex}}^\Omega = \sum_{i=1}^N \frac{\Gamma_i}{2}(x_i^2 + y_i^2).
\]
These expressions coincide up to sign with those in \cite{AdRa1988},
which is a consequence of the fact that our symplectic structure is
the negative of theirs.

Finally, for the stream function term $\mathbf{d} \Theta$ in $\beta _\mu$,
observe that $\Theta$ is an $\operatorname{SE}(2) $-invariant one-form on $\operatorname{SE}(2) \times \mathbb{R}^{2N}$.  The $\Bg$-potential associated to $\mathbf{d} \Theta$ is hence
given by (see the remark following Theorem~7.1.1 in
\cite{MarsdenHamRed}) 
\[
  \left<\phi_{\mathrm{stream}}, \xi\right> = \mathbf{i}_{\xi_{\operatorname{SE}(2) \times \mathbb{R}^{2N}}} \Theta
  \quad \text{for all } \xi \in \mathfrak{se}(2).
\]
Explicitly,
\[
\begin{pmatrix}
  \phi_{\mathrm{stream}}^x \\
  \phi_{\mathrm{stream}}^y 
\end{pmatrix} = 
\begin{pmatrix}
  \sum_{i = 1}^N \Gamma_i \Psi_x(g, \mathbf{x}_i) \\
  \sum_{i = 1}^N \Gamma_i \Psi_y(g, \mathbf{x}_i)
\end{pmatrix} \quad \text{and} \quad 
\phi_{\mathrm{stream}}^\Omega = \sum_{i = 1}^N \Gamma_i \Psi_\omega(g,
\mathbf{x}_i), 
\]
where the elementary stream functions are given by (\ref{stream}).

The full momentum map $J$ is then the sum of these three
contributions: 
\[
  J = J_{\mathrm{can}} - \phi_{\mathrm{vortex}} - \phi_{\mathrm{stream}}, 
\]
and this is precisely (\ref{fullmommmap}).
\end{proof}

The expression for the momentum map can be rewritten by introducing
the momentum map of the solid-fluid system in \textbf{\emph{body coordinates}}
(\cite{Sh2005, KaOs08}):
\begin{align*}
\begin{pmatrix}
J_X \\ J_Y
\end{pmatrix}
& =  
\begin{pmatrix}
  \Pi_x \\ \Pi_y 
\end{pmatrix}
+ \sum_{i=1}^N \Gamma_i \mathbf{X}_i \times \mathbf{e}_3 - \sum_{i=1}^N
\Gamma_i 
\begin{pmatrix}
  \Psi_X(\mathbf{X}_i) \\ \Psi_Y(\mathbf{X}_i)
\end{pmatrix}
\\
J_\Omega & = \Pi_\Omega - \sum_{i=1}^N \frac{\Gamma_i}{2}
\left\Vert\mathbf{X}_i\right\Vert^2 - \sum_{i=1}^N\Gamma_i \Psi_\Omega(\mathbf{X}_i),
\end{align*}
where $\Psi_X, \Psi_Y, \Psi_\Omega$ are the expression for the
elementary stream functions in body coordinates.  By using the
relation (\ref{transition}) between inertial and body coordinates, one
can show that the spatial and body momentum are related by 
\begin{equation} \label{relation}
\begin{pmatrix}
J_x \\ J_y
\end{pmatrix}
= R
\begin{pmatrix}
J_X \\ J_Y
\end{pmatrix}
+ \Gamma \mathbf{x}_0 \times \mathbf{e}_3 
\quad \text{and} \quad  
J_\omega = J_\Omega + (\mathbf{x}_0 \times \mathbf{J}_{\mathbf{x}})\cdot \mathbf{e}_3,
\end{equation}
where $\mathbf{J}_{\mathbf{x}} = ( J_x, J_y)^T$.  Note that the spatial
momentum is a function on $\operatorname{SE}(2)  \times \mathbb{R}^{2N}$, whereas the body
momentum is a function on $\mathbb{R}^{2N}$, and that the
$\operatorname{SE}(2) $-dependence of $J$ is determined by the relation above.  In
particular, if we evaluate $J$ at the identity, putting $R =
\mathbf{1}$ and $\mathbf{x}_0 = 0$ in (\ref{relation}), then we obtain
just the body momentum:
\[
  J(e, \mathbf{x}) = J(\mathbf{X}).
\]
\paragraph{Non-equivariance of the Momentum Map.}
Using the definition (\ref{twococycle}), the non-equivariance two-cocycle of the
momentum map is a map $\Sigma: \mathfrak{se}(2) \times \mathfrak{se}(2)
\rightarrow C^\infty(\operatorname{SE}(2)  \times \mathbb{R}^{2N})$.  To compute $\Sigma$,
we recall the basis $\{\mathbf{e} _\Omega, \mathbf{e} _x, \mathbf{e}_y\}$ of 
$\mathfrak{se}(2)$ given in (\ref{sebasis})  and let $\{\mathbf{e}^{\ast} _\Omega, \mathbf{e}^{\ast} _x, \mathbf{e}^{\ast} _y\}$ be
the corresponding dual basis of $\mathfrak{se}(2)^{\ast} $.

\begin{theorem}
  The non-equivariance two-cocycle $\Sigma$ of the momentum map $J$ is
  given by 
  \[
    \Sigma = - \Gamma \mathbf{e}^{\ast} _x \wedge \mathbf{e}^{\ast} _y,
  \]
  where $\Gamma = \sum_{i=1}^N \Gamma_i$ is the total vortex strength.
\end{theorem}
\begin{proof}
  In order to use (\ref{twococycle}) to compute $\Sigma$, we need the
  infinitesimal generators corresponding to the basis elements
  $\mathbf{e}_x$, $\mathbf{e}_y$ and $\mathbf{e}_\Omega$.  The infinitesimal
  generator of any element $\xi \in \mathfrak{se}(2)$, evaluated at $(e;
  \mathbf{x}_1, \ldots, \mathbf{x}_N)$, will be denoted by $\tilde{\xi}$ and 
is given by
  \[
    \tilde{\xi} := \xi_{\operatorname{SE}(2)  \times \mathbb{R}^{2N}}(e; \mathbf{x}_1, \ldots, \mathbf{x}_N) =
    (\xi; \xi_{\mathbb{R}^{2N}}(\mathbf{x}_1), \ldots, \xi_{\mathbb{R}^{2N}}(\mathbf{x}_N)),
  \]
  where $\xi_{\mathbb{R}^{2N}}$ is the infinitesimal generator of $\xi$
  for the fundamental action of $\operatorname{SE}(2) $ on $\mathbb{R}^2$.  Explicitly,
  \[
    (\mathbf{e}_\Omega)_{\mathbb{R}^2} = -y \frac{\partial}{\partial x} +
    x\frac{\partial}{\partial y}, \quad
    (\mathbf{e}_x)_{\mathbb{R}^2} = \frac{\partial}{\partial x}, \quad \text{and} \quad  (\mathbf{e}_y)_{\mathbb{R}^2} =
    \frac{\partial}{\partial y}. 
  \]

  Evaluating $\beta_\mu$ on these vectors, we obtain
  \begin{align*}
    \beta_\mu(\tilde{\mathbf{e}}_x, \tilde{\mathbf{e}}_y) & = -\Gamma, \\
    \beta_\mu(\tilde{\mathbf{e}}_x, \tilde{\mathbf{e}}_\Omega) & = \sum_{i=1}^N \Gamma_i
    \left(-\mathbf{X}_i + \frac{1}{\left\Vert\mathbf{X}_i\right\Vert^2} \right), \\
    \beta_\mu(\tilde{\mathbf{e}}_y, \tilde{\mathbf{e}}_\Omega) & = \sum_{i=1}^N \Gamma_i
    \left(-\mathbf{Y}_i + \frac{1}{\left\Vert\mathbf{Y}_i\right\Vert^2} \right).
  \end{align*}

  Remarkably, when calculating $\Sigma$, the last two expressions are
  cancelled entirely by opposite contributions from the remaining terms:
  \begin{align*}
    \Sigma(\mathbf{e}_x, \mathbf{e}_\Omega) & = -\left<\phi, [\mathbf{e}_x,
      \mathbf{e}_\Omega]\right> + \beta_\mu(\tilde{\mathbf{e}}_x,
    \tilde{\mathbf{e}}_\Omega) \\
    & = \phi_y + \beta_\mu(\tilde{\mathbf{e}}_x,
    \tilde{\mathbf{e}}_\Omega) \\
    & = 0,
  \end{align*}
  and similar for $\Sigma(\mathbf{e}_y, \mathbf{e}_\Omega)$.  The only
  non-zero term is given by $\Sigma(\mathbf{e}_x, \mathbf{e}_y) = - \Gamma$.
\end{proof}

In its current form, $\Sigma$ has the same form as the
non-equivariance two-cocycle for the $N$-vortex problem in an
unbounded fluid (see \cite{AdRa1988}).

\subsection{Poisson Structures}

Now we come to the conclusion of this paper.   Using the theory developed in the previous sections, we derive an explicit form for the reduced Poisson bracket on $\mathfrak{se}(2)^{\ast}  \times \mathbb{R}^{2N}$ associated to the magnetic symplectic form (\ref{magsympexpl}): this will turn out to be the BMR bracket.  In the terminology of section~\ref{sec:gen}, this is the bracket 
$\{\cdot, \cdot\}_{\mathrm{int}}$ with interaction terms.  Secondly, we then use the shift map $\hat{\mathcal{S}}$ 
(see (\ref{hatshift})) associated to the momentum map (\ref{fullmommmap}) to obtain the Poisson structure 
$\{\cdot, \cdot\}_\Sigma$ where the interaction terms are absent, at the expense of a non-equivariance cocycle.  When made more explicit, the latter bracket turns out to be the SMBK bracket.

\paragraph{The BMR Poisson Structure.}  The bracket obtained by Poisson reduction of the magnetic symplectic structure on $T^\ast \operatorname{SE}(2) \times \mathbb{R}^{2N}$ was described in theorem~\ref{thm:bmr}.  
The explicit computation of this Poisson
bracket boils down to substituting the explicit
expresssion (\ref{fullmommmap}) for the momentum map into the Poisson
bracket $\{\cdot, \cdot\}$ in theorem~\ref{thm:bmr}.  After a long,
but straightforward calculation, one obtains the BMR bracket
(\ref{bmrbracket}).  As an illustration, we compute one bracket
element, leaving the others to the reader.

We consider the functions $F = \Pi_x$ and $G = \Pi_y$ on $\mathfrak{se}(2)^{\ast} 
\times \mathbb{R}^{2N}$ and compute $\{\Pi_x, \Pi_y\}_{\mathrm{int}}$.  Note that,
considered as elements of $\mathfrak{se}(2)$, 
\[
\frac{\delta F}{\delta \mu} = \mathbf{e}_x \quad \text{and} \quad  \frac{\delta G}{\delta
  \mu} = \mathbf{e}_y.
\]

Computing (\ref{bmrbracket}) term by term, we have first of all that
\begin{equation*}
  \Pi_x \star \Pi_y   = \{\phi_x, \phi_y\}_{\mathrm{vortex}} \nonumber \\
      = - \sum_{i = 1}^N \Gamma_i \frac{\left\Vert\mathbf{X}_i\right\Vert^4 -
      R^4}{\left\Vert\mathbf{X}_i\right\Vert^4}, 
\end{equation*}
where $\{\cdot, \cdot\}_{\mathrm{vortex}}$ is the vortex bracket defined in (\ref{vortexbracket}).  Secondly, 
\[
  \beta_\mu(\tilde{\mathbf{e}}_x, \tilde{\mathbf{e}}_y) = -\Gamma,
\]
while the other terms are zero. Finally, it follows that 
\[
  \{\Pi_x, \Pi_y\}_{\mathrm{int}} = \Gamma - \sum_{i = 1}^N \Gamma_i \frac{\left\Vert\mathbf{X}_i\right\Vert^4 -
      R^4}{\left\Vert\mathbf{X}_i\right\Vert^4},
\]
which is precisely the first bracket element of (\ref{bmrbracket}).
The computation of the other elements is similar.

\paragraph{The SMBK Poisson Structure.}

By subjecting the BMR Poisson structure
to the shift map $\hat{\mathcal{S} }$, we can eliminate the interaction terms
from the Poisson structure, at the expense of introducing a non-equivariance cocycle.
As shown in
Theorem~\ref{thm:poissonsimple}, the result is a Poisson structure
consisting of the sum of the Lie-Poisson and the vortex Poisson
structure on the individual factors, together with a cocycle term:
\[
\{F, G\}_\Sigma = \{F_{|\mathfrak{se}(2)^{\ast} },
G_{|\mathfrak{se}(2)^{\ast} }\}_{\mathrm{l.p}} + \{F_{|\mathbb{R}^{2N}},
G_{|\mathbb{R}^{2N}}\}_{\mathrm{vortex}} - 
\Sigma \left( 
\frac{\delta F}{\delta \mu}, \frac{\delta F}{\delta \mu} \right).
\]
The last term is explicitly given by 
\[
\Sigma \left( 
\frac{\delta F}{\delta \mu}, \frac{\delta F}{\delta \mu} \right) = - \Gamma \left( \frac{\partial
    F}{\partial \Pi_x}\frac{\partial
    G}{\partial \Pi_y} - \frac{\partial
    G}{\partial \Pi_x}\frac{\partial
    F}{\partial \Pi_y} \right).
\]

In the case where the total vorticity $\Gamma$ is zero, 
this term vanishes and 
the bracket
reduces to the SMBK bracket (\ref{thm:poissonsimple}).  

%

\paragraph{The Symplectic Leaves.}  In the case where the two-cocycle
$\Sigma$ vanishes, or equivalently, when the total circulation
$\Gamma$ is zero, a convenient expression for the symplectic leaves in 
$\mathfrak{se}(2)^{\ast}  \times \mathbb{R}^{2N}$ can be read off from
proposition~\ref{prop:sympleaf}.  

Recall that the symplectic leaves for the Lie-Poisson structure in
$\mathfrak{se}(2)^{\ast}  \cong \mathbb{R}^3$ come in two varieties.  One class
consists of cylinders whose axis is the $\Omega$-axis:
\[
  \mathcal{O} = \{ (\Pi_x, \Pi_y, \Omega) : \Pi_x^2 + \Pi_y^2 =
  \text{constant} \}, 
\]
while the other class consists of the individual points $(0, 0,
\Omega)$ of the $\Omega$-axis.  The symplectic leaves in
$\mathfrak{se}(2)^{\ast}  \times \mathbb{R}^{2N}$ are then the product of the
symplectic leaves of the Lie-Poisson structure with $\mathbb{R}^{2N}$.

\section{Conclusions and Future Work}

In this paper, we have 
used reduction theory to give a
systematic derivation of the equations of motion for a circular rigid body in a perfect fluid with point vortices.  Among other things, we have derived the Poisson structures that govern this problem, and related them to the Poisson structures in the literature.  However, the usefulness of our geometric method is not limited to merely reproducing the correct Poisson structure for one specific problem:  by uncovering fundamental geometric structures such as the Neumann connection, we have shed new light on the precise nature of the solid-fluid interaction.   Moreover, even though a number of simplifying assumptions were made in the beginning, 
it is expected that the present method can be extended without too much trouble to cover more general cases.  
Below, we have listed a number of open questions for which our method seems especially suited.

\paragraph{Non-zero Circulation.}  As we have seen in this paper, if
the total strength of the point vortices is non-zero, a cocycle
term appears in the equations of motion.  In a previous paper (see \cite{VaKaMa08}), we
studied the case of a rigid body moving in a potential flow with
circulation and found that the circulation manifests itself through
the curvature of the Neumann connection.  

From a physical point of view, both phenomena are very similar.
Indeed, having a non-zero circulation around the rigid body
amounts 
to placing a vortex at the center of the rigid body.
Hence, mathematically speaking, there appears to be a duality between
the description in terms of cocycles and in terms of curvatures.
Having a better mathematical understanding of this duality would prove
to be useful for other physical theories as well, for example the
dynamics of spin glasses in \cite{CeMaRa2004}.  In addition, it would be interesting to investigate the link between this theory with the work of \cite{GaRa2008}.

\paragraph{Arbitrary Body Shapes.}  The pioneering work of SMBK was
extended to cover the case of rigid bodies with arbitrary shapes by
\cite{Sh2005}, and studied further by \cite{KaOs08} with a view
towards stability and bio-locomotion.

By using methods of complex analysis, the methods of the present paper can
easily be extended to cover the case of rigid bodies of arbitrary
shape as well: any closed curve enclosing a simply-connected area of
the plane can be mapped to the unit circle by a suitable conformal
transformation and by pulling back the geometric objects in this paper
along that map, arbitrary body shapes can be treated.

Moreover, the objects defined in this paper transform naturally under
conformal transformations.  The magnetic form $\beta_\mu$, for
instance, will be mapped into a form $\beta_\mu'$ which has the same
appearance as before, but with now $\Psi_A$ the elementary stream
functions for a body of that shape.  As known from the classical fluid
dynamics literature, these stream functions also transform naturally
under conformal mappings.

\paragraph{Three-dimensional Bodies.}  

In contrast to calculations using only standard, ad-hoc methods, our
setup naturally generalizes to the case of three-dimensional flows
interacting with rigid bodies.  
\cite{ShShKeMa08} consider a rigid body interacting with
\textbf{\emph{vortex rings}}.  The space of point vortices $\mathbb{R}^{2N}$ then
is replaced by the space $\mathcal{M}$ of $N$ vortex rings, and much of the
analysis in the 2D case carries through, at the expense of significant
computational work.

From a geometric point of view, however, $\mathcal{M}$ plays a similar role
as $\mathbb{R}^{2N}$: both are coadjoint orbits of the group $\mathrm{Diff}_{\mathrm{vol}}$ of
volume-preserving diffeomorphisms.  Hence, $\mathcal{M}$ is equipped with a
natural symplectic structure (see \cite{ArnoldKhesin}), which is of
importance for the calculation of $\beta_\mu$ in this case.  Also, the definition of the 
Neumann connection remains valid in three dimensions. 
 We therefore expect that the geometric approach 
can be extended without too much difficulty to this case, and 
will lead to a
conceptually much clearer picture.

\paragraph{Routhian Reduction by Stages and Dirac Reduction.}

It has been known for some time that the Lagrangian analogue of
symplectic reduction is \textbf{\emph{Routhian reduction}} (\cite{Routh00}).  By
starting from the Lagrangian of an incompressible fluid together with
the rigid body Lagrangian, one could perform reduction on the
Lagrangian side using this theory.  The advantage would be
that Routhian reduction preserves the variational nature of the
system.
However, as the reduction procedure for the fluid-solid system
consists of two succesive reductions, one would need a theory of
\textbf{\emph{Routhian reduction by stages}}.  Developing such a theory would
be of considerable interest.

A related approach consists of using Dirac structures.   Through the Hamilton-Pontryagin principle, Dirac structures offer a way of incorporating both the Lagrangian and Hamiltonian formalisms making them ideally suited for systems with degenerate Lagrangians; see \cite{MarsdenDirac}.
Since the point vortex Lagrangian is degenerate, Dirac structures are likely to be useful here.  

\appendix
\section{Further Properties of the Neumann Connection}

Even though the Neumann connection is used in many problems from
geometric fluid dynamics or differential geometry, a systematic
presentation of its properties seems to be lacking.  In this appendix
we prove some basic theorems related to the Neumann connection.

\subsection{Fiber Bundles and Connections.}

Recall that if $\pi: Q \rightarrow Q/G$ is a principal fibre bundle
with structure group $G$, then a connection one-form on $Q$ is a
$\mathfrak{g}$-valued one-form $\mathcal{A}$ on $Q$ satisfying the following two
properties:
\begin{enumerate}
  \item $\mathcal{A}$ is \textbf{\emph{$G$-equivariant}}: $\sigma^{\ast} _g \mathcal{A} = \mathrm{Ad}_{g^{-1}}
    \circ \mathcal{A}$, where $\sigma_g : Q \rightarrow Q$ denotes the
    $G$-action, and $\mathrm{Ad}$ is the adjoint action on $\mathfrak{g}$.
  \item Let $\xi_Q$ be the \textbf{\emph{infinitesimal generator}} associated
    to an element $\xi \in \mathfrak{g}$, \emph{i.e.}  
    \[
       \xi_Q(q) = \frac{d}{dt} \sigma(\exp t\xi, q) \Big|_{t = 0}, \quad
       \text{where $\sigma(g, q) = \sigma_g(q)$}.
    \]
    Then $\mathcal{A}(\xi_Q) = \xi$.
\end{enumerate}

At each point $q \in Q$, the connection $\mathcal{A}$ induces complementary
projection operators $P_V, P_H: TQ \rightarrow TQ$, given by 
\[
  P_V(v_q) = [\mathcal{A}(v_q)]_Q(q), \quad \text{and} \quad  P_H = 1 - P_V,
\]
and referred to as the \textbf{\emph{vertical}} and \textbf{\emph{horizontal}}
projection operators, respectively.  
More information on principal fiber bundles and connections can be
found in \cite{KN1}.

\paragraph{The Curvature.}

The \textbf{\emph{curvature}} of a principal fiber bundle connection $\mathcal{A}$ is
the $\mathfrak{g}$-valued two-form $\mathcal{B}$ defined as follows.  Let $u_q,
v_q \in T_q Q$ and consider vector fields $X_u$ and $X_v$ extending
$u_q$ and $v_q$, \emph{i.e.} such that $X_u(q) = u_q$ and $X_v(q) =
v_q$.  Denote the horizontal part of $X_u$ as $X_u^H := P_H \circ
X_u$, and similar for $X_v^H$.  Then, $\mathcal{B}$ can be conveniently
expressed as 
\begin{equation} \label{curvdef}
   \mathcal{B}_q(u_q, v_q) = - \mathcal{A}_q([X^H_u, Y^H_v]).
\end{equation}
With this definition, $\mathcal{B}$ is $G$-equivariant and vanishes on
vertical vector fields:
\begin{equation} \label{props}
  \sigma^{\ast} _g \mathcal{B} = \mathrm{Ad}_{g^{-1}} \mathcal{B} \quad \text{and} \quad  i_{\xi_Q} \mathcal{B} = 0, 
\end{equation}
for all $g \in G$ and $\xi \in \mathfrak{g}$.

\subsection{The Neumann Connection}

\paragraph{The Connection One-form.}

We recall from section~\ref{sec:neumann} that the connection one-form for the Neumann connection is given by 
\[
  \mathcal{A}_{(\varphi, g)}(\dot{g}, \dot{\varphi}) = \varphi^{\ast} 
  \mathbf{u}_{\rm v},
\]
where $\mathbf{u}_{\rm v}$ is the divergence-free part in the
Helmholtz-Hodge decomposition (\ref{hhdecomp}) of the Eulerian
velocity $\mathbf{u}$; see  (\ref{connection}).

Before showing that this prescription
 indeed yields
a well-defined connection form, we note that each divergence-free
vector field $\mathbf{u} \in \mathfrak{X}_{\mathrm{vol}}$ defines a vertical vector field on
$Q$, denoted by $\mathbf{u}_Q$ and given by
\[
  \mathbf{u}_Q(g, \varphi) = (0, T \varphi \circ \mathbf{u}).
\]
The vector field $\mathbf{u}_Q$ is the infinitesimal generator
associated to $\mathbf{u}$ under the action of $\mathrm{Diff}_{\mathrm{vol}}$ on $Q$.

\begin{proposition} \label{prop:conn}
  The one-form $\mathcal{A}$ defined in (\ref{connection}) is a connection
  one-form.  In other words, 
  \begin{enumerate}
    \item $\mathcal{A}$ is $\mathrm{Diff}_{\mathrm{vol}}$-equivariant: for all $\phi \in
      \mathrm{Diff}_{\mathrm{vol}}$ and $(g, \dot{g}; \varphi, \dot{\varphi}) \in TQ$, 
      \[
          \mathcal{A}_{(g, \varphi \circ \phi)}(\dot{g}, \dot{\varphi} \circ \phi) = 
          \phi^{\ast}  \mathcal{A}_{(g, \varphi)}(\dot{g}, \dot{\varphi});
      \]
    \item If $\mathbf{u}$ is an element of $\mathfrak{X}_{\mathrm{vol}}$ with associated
      fundamental vector field $\mathbf{u}_Q$, then $\mathcal{A}(\mathbf{u}_Q) =
      \mathbf{u}$.
  \end{enumerate}
\end{proposition}
\begin{proof}
  To prove equivariance, note that the Eulerian velocity associated to
  $\dot{\varphi} \circ \phi$ is equal to that associated to
  $\dot{\varphi}$:
  \[
    \mathbf{u}' = (\dot{\varphi} \circ \phi) \circ (\varphi \circ
    \phi)^{-1} = \dot{\varphi} \circ \varphi^{-1} = \mathbf{u}.
  \]
  Hence, 
  \[
    \mathcal{A}_{(g, \varphi \circ \phi)}(\dot{g}, \dot{\varphi} \circ \phi)
    = (\varphi \circ \phi)^{\ast}  \mathbf{u}_{\rm v} 
    = \phi^{\ast}  (\varphi^{\ast}  \mathbf{u}_{\rm v})
    = \phi^{\ast}  \mathcal{A}_{(g, \varphi)}(\dot{g}, \dot{\varphi}).
  \]

  To prove the second property, note that the push-forward $X =
  \varphi_\ast \mathbf{u}$ is divergence-free and tangent to the boundary
  of $\mathcal{F}$ since $\mathbf{u}$ is divergence-free and tangent to
  $\partial \mathcal{F}_0$ and $\varphi$ is volume-preserving.

  Evaluating the connection one-form on $\mathbf{u}$ therefore gives 
  \[
    \mathcal{A}_{(g, \varphi)}(\mathbf{u}_Q(g, \varphi)) = \varphi^{\ast}  X = \mathbf{u},
  \]
  which concludes the proof that $\mathcal{A}$ is a well-defined connection
  one-form. 
\end{proof}

\paragraph{The Horizontal Lift.}

The horizontal lift operator associates to each $(g, \varphi)$ a
linear map $\mathbf{h}_{(g, \varphi)}: T_g \operatorname{SE}(2)  \rightarrow T_{(g, \varphi)}
Q$ with the following properties: 
\begin{enumerate}
  \item The composition $T_{(g, \varphi)} \pi \circ \mathbf{h}_{(g, \varphi)}$
    is the identity in $T_g \operatorname{SE}(2) $; 
  \item The image of $\mathbf{h}_{(g, \varphi)}$ consists of horizontal
    vectors: $\mathcal{A}_{(g, \varphi)} \circ \mathbf{h}_{(g, \varphi)} = 0$. 
\end{enumerate}

For the Neumann connection, the horizontal lift
is given by 
\[
   \mathbf{h}_{(g, \varphi)}(\dot{g}) = \left(\dot{g}, \nabla \Phi \circ \varphi\right),
\]
where $\Phi$ is the solution of the Neumann problem (\ref{neumann})
associated to $(g, \dot{g})$; \emph{i.e.} $(\omega, \mathbf{v}) =
\dot{g}g^{-1}$.    It is easy to show that $\mathbf{h}$ satisfies both properties of the horizontal lift, and hence that $\mathbf{h}$ is indeed the horizontal lift of the Neumann connection as we claimed in section~\ref{sec:neumann}.

\paragraph{The Neumann Connection as a Mechanical Connection.}

When given a group-invariant metric on the total space of a principal
fiber bundle, the \textbf{\emph{mechanical connection}} is defined by
declaring its horizontal subspaces to be orthogonal to the vertical
bundle.  The terminology stems from the fact that the underlying
metric is usually the kinetic-energy metric of a mechanical system.

Since there exists a natural metric (\ref{metric}) on $Q$, it should
therefore come as no surprise that the Neumann connection can also be
viewed as a mechanical connection.  

The horizontal subspace of the Neumann connection at an element
$(g, \varphi) \in Q$ is given by the kernel of $\mathcal{A}_{(g, \varphi)}$.
Explicitly, a tangent vector $(g, \dot{g}; \varphi, \dot{\varphi})$ is
in the kernel of $\mathcal{A}_{(g, \varphi)}$ if and only if $\dot{\varphi}
= \nabla \Phi \circ \varphi$, where $\Phi$ is the solution of the
Neumann problem (\ref{neumann}) associated to $(g, \dot{g})$.

On the other hand, the vertical subspace at $(g, \varphi)$ is
generated by elements of the form $\mathbf{u}_Q(g, \varphi)$.  Verifying
that the Neumann connection is indeed the mechanical connection now
boils down to checking that the vertical and horizontal subspaces are
orthogonal with respect to the metric (\ref{metric}).

Let $(g, \dot{g}; \varphi, \nabla \Phi \circ \varphi)$ and
$\mathbf{u}_Q(g, \varphi)$ be horizontal and vertical tangent vectors,
respectively.  Then we have
\begin{multline*}
  \left\langle \! \left\langle(g, \dot{g}; \varphi, \nabla \Phi \circ \varphi),
  \mathbf{u}_Q(g, \varphi)\right\rangle \!
  \right\rangle  = \left< \nabla \Phi \circ \varphi,
    T\varphi \circ \mathbf{u} \right>_{\operatorname{Emb}} \\
   = \int_{\mathcal{F}_0} (\nabla \Phi \circ \varphi) \cdot (T\varphi \circ
  \mathbf{u}) \, \eta_0 
   = \int_{\mathcal{F}} \nabla \Phi \cdot (\varphi_\ast \mathbf{u}) \, \eta 
   = 0,
\end{multline*}
where the last equality follows from the $L_2$-orthogonality between
gradient vector fields and divergence-free vector fields tangent to
the boundary of $\mathcal{F}$.

\section{Proofs of the Theorems in Section~\ref{sec:gen}} \label{appB}

In this technical appendix, we provide proofs for the theorems in section~\ref{sec:gen} regarding the reduced Poisson structures with or without interaction terms.   We do so by using the diffeomorphism of $T^\ast H \times P$ with 
$H \times \mathfrak{h} ^{\ast}  \times
P$ given by left translation.  
More specifically, 
we introduce a diffeomorphism $\Lambda : T^{\ast}  H \times P
\rightarrow H \times \mathfrak{h} ^{\ast}  \times P$ by
\[
\Lambda(g, \alpha_g; p) = (g, T^{\ast}  L_g (\alpha_g), g^{-1}
  \cdot p),
\]
and a diffeomorphism $\lambda : H \times P \rightarrow H \times P$ by
$\lambda(g, p) = (g, g^{-1} \cdot p)$ such that the following diagram
commutes: 
\[
\xymatrix{ 
  T^{\ast}  H \times P \ar[r]^\Lambda \ar[d]_{\pi_H} & H \times
  \mathfrak{h} ^{\ast}  \times P \ar[d]^{\mathrm{pr}} \\
  H \times P \ar[r]_\lambda & H \times P}
\]
Here, $\pi_H: T^{\ast}  H \times P \rightarrow H \times P$ is given by
$\pi_H(g, \alpha_g; p) = (g, p)$, while $\mathrm{pr}$ simply forgets the
second factor: $\mathrm{pr}(g, \mu, x) = (g, x)$.

The advantage of this construction is that under $\Lambda$
the action (\ref{diagaction}) becomes left translation on the first factor.
That is, if
$\Lambda(\alpha_g; p) = (g, \mu, x)$, then $\Lambda(h \cdot(\alpha_g;
p)) = (hg, \mu, x)$.  The fact that the diagonal action
(\ref{diagaction}) reduces in this way greatly simplifies Poisson reduction.

\subsection{Symplectic Forms in Body Coordinates.}

We start by deriving in lemma~\ref{lemma:sympbody} an explicit expression for the push-forward
$\Lambda_\ast \Omega$ of the symplectic form.  We then define a 
natural \textbf{\emph{shift map}} $\mathcal{S} $ from $H
\times \mathfrak{h} ^{\ast}  \times P$ to itself, given by
\[
  \mathcal{S} (g, \mu, x) = (g, \mu - \phi(e, x), x),  
\]
where $\phi: H \times P \rightarrow \mathfrak{h} ^{\ast} $ is the
$\Bg$-potential associated to the action of $H$ on $H \times P$,
\emph{i.e.} the solution of (\ref{bgpot}).  This map takes the
symplectic structure (\ref{bigsympstruct}) with interaction terms of lemma~\ref{lemma:sympbody} into
one where the interaction terms are absent.  The explicit form of the resulting symplectic form is derived in lemma~\ref{lemma:pf}.

For future reference, we will refer to the push-forwards $\Lambda_\ast \omega_{\mathrm{can}}$ 
and $\Lambda_\ast \Omega$
as
$\omega_{{\mathrm{body}}}$ and $\Omega_{{\mathrm{body}}}$, respectively:
\begin{equation} \label{cushman}
\omega_{{\mathrm{body}}} = \Lambda_\ast \omega_{\mathrm{can}} \quad \text{and} \quad  
\Omega_{{\mathrm{body}}} = \Lambda_\ast \Omega.
\end{equation}

\begin{lemma} \label{lemma:sympbody}
  The push-forward $\Lambda_\ast \Omega$ of the symplectic form
  $\Omega$ to $H \times \mathfrak{h} ^{\ast}  \times P$ is given by 
\begin{multline}
  (\Lambda_\ast \Omega) (g, \mu, x)((g\cdot \xi, \rho, v_x), (g\cdot
  \eta, \sigma, w_x)) = \\
  - \left<\rho, \eta\right> + \left<\sigma, \xi\right> + \left<\mu, [\xi, \eta]\right>
 \label{bigsympstruct} \\
  - B(e, x)(\xi_{H\times P}(e, x), \eta_{H\times P}(e, x)) 
   - B(e, x)((0, v_x), (0, w_x)) \\
  - \left<\mathbf{d} \phi_\xi(e, x), (0, w_x)\right> + \left<\mathbf{d} \phi_\eta(e, x),
   (0, v_x)\right>,  
\end{multline}
where $(g, \mu, x)$ is an element of $H \times \mathfrak{h} ^{\ast}  \times P$
and $(g\cdot \xi, \rho, v_x)$, $(g\cdot \eta, \sigma, w_x)$ are
elements of $T_{(g, \mu, x)}(H \times \mathfrak{h} ^{\ast}  \times P) \cong T_g
H \times \mathfrak{h} ^{\ast}  \times T_x P$.
\end{lemma}
\begin{proof}
Notice first that 
$\Lambda_\ast \Omega = \Lambda_\ast \omega_{\mathrm{can}} - \mathrm{pr}^{\ast} (
  \lambda_\ast B)$.
The first term is the expression for the canonical symplectic form in
body coordinates calculated by R. Cushman (quoted in \cite{AbMa78},
proposition~4.4.1) and is given by
\[
(\Lambda_\ast \omega_{\mathrm{can}}) (g, \mu, x)((g\cdot \xi, \rho, v_x), (g\cdot
  \eta, \sigma, w_x)) =  - \left<\rho, \eta\right> + \left<\sigma, \xi\right> +
  \left<\mu, [\xi, \eta]\right>. 
\]

For the second term, we have
\begin{align*}
(\lambda_\ast& B) (g, x)((g\cdot\xi, v_x), (g\cdot\eta, w_x))   \\
  & =B(\lambda^{-1}(g, x))(T\lambda^{-1}(g\cdot\xi, v_x),
  T\lambda^{-1}(g\cdot\eta, w_x)) \\
  & =B(g, g \cdot x)\big((g\cdot\xi, T\Psi_g(v_x) + T\hat{\Psi}_x(g \cdot
  \xi)), (g\cdot\eta, T\Psi_g(w_x) + T\hat{\Psi}_x(g \cdot
  \eta))\big),
\end{align*}
where $\Psi_g: P \rightarrow P$ denotes the action of $H$ on $P$:
$\Psi_g(x) = g\cdot x$, and $\hat{\Psi}_x: G \rightarrow P$ is the map
defined by $\hat{\Psi}_x(g) = g \cdot x$.  Note that $T\hat{\Psi}_x(g
\cdot \xi) = g \cdot \xi_P(p)$.  Taking this into account, as well as
the $H$-invariance of $B$, allows us to rewrite the above expression
as
\begin{align*}
  B(e, x)& \big((\xi, v_x + \xi_P(x)), (\eta, w_x + \eta_P(x)) \big)
  \\
  & = B(e, x)((0, v_x), (0, w_x)) + B(e, x)((\xi, \xi_P(x)), (\eta,
  \eta_P(x))) \\
  & + B(e, x)((0, v_x), (\eta, \eta_P(x))) + B(e, x)((\xi, \xi_P(x)),
  (0, w_x)).
\end{align*}
The third term can be rewritten as
\[ 
B(e, x)((0, v_x), (\eta, \eta_P(x))) = - \mathbf{d} \phi_\eta(e, x)(0, v_x), 
\]
and similar for the last term.  Putting everything together, we obtain
(\ref{bigsympstruct}).
\end{proof}

\begin{lemma} \label{lemma:pf}
  The push-forward $\mathcal{S} _\ast \Omega_{\mathrm{body}}$ on $H \times
  \mathfrak{h} ^{\ast}  \times P$ is again a symplectic form, given by 
  \begin{multline} \label{pfomega}
    (\mathcal{S} _\ast \Omega_{\mathrm{body}})(g, \mu, x)((g\cdot \xi, \rho, v_x),
    (g\cdot
    \eta, \sigma, w_x)) = \\
    \omega_{\mathrm{body}}(g, \mu)((g \cdot \xi, \rho), (g\cdot
    \eta, \sigma))
    - \omega_P(v_x, w_x) - \Sigma(\xi, \eta).
  \end{multline}
  Here, the arguments of $\mathcal{S} _\ast \Omega_{\mathrm{body}}$ have the same
  meaning as in lemma~\ref{lemma:sympbody}, and $\Sigma$ is the
  non-equivariance two-cocycle (\ref{twococycle}) of $\phi$.
\end{lemma}
\begin{proof}
The pushforward of $\omega_{\mathrm{body}}$ under $\mathcal{S} $ is given by 
\begin{multline} \label{pushforwardcushman}
  (\mathcal{S} _\ast \omega_{\mathrm{body}})(g, \mu, x)((g \cdot \xi, \rho, v_x),
  (g\cdot \eta, \sigma,w_x)) = \\ \omega_{\mathrm{body}}(g, \mu', x) ((g \cdot
  \xi, \rho', v_x), (g\cdot
  \eta, \sigma' ,w_x)) = \\
  - \left<\rho', \eta\right> + \left<\sigma', \xi\right> + \left<\mu', [\xi,
    \eta]\right>,
\end{multline}
where 
\[
  \mu' = \mu + \phi(e, p), \quad \rho' = \rho + T\phi(0, v_x) \quad \text{and} \quad 
  \sigma' = \sigma + T\phi(0, w_x).
\]
Substituting this into (\ref{pushforwardcushman}), we get 
\begin{align*}
  \left<\rho', \eta\right> & = \left<\rho, \eta\right> + \left<T\phi(0, v_x),
    \eta\right> \\
   & = \left<\rho, \eta\right> + \mathbf{d} \phi_\eta (0, v_x),
\end{align*}
and similar for $\left<\sigma', \xi\right>$.  For the term involving $\mu'$,
we obtain
\begin{align*}
  \left<\mu', [\xi, \eta]\right> & = \left<\mu, [\xi, \eta]\right> + \left<\phi(e,
    x), [\xi, \eta]\right> \\
   & = \left<\mu, [\xi, \eta]\right> - \Sigma(\xi, \eta) + B(e, x)(\xi_{H
     \times P}(e, x), \eta_{H
     \times P}(e, x) ), 
\end{align*}
where we have used the definition of $\Sigma$.  Hence, 
\begin{multline*}
  (\mathcal{S} _\ast \omega_{\mathrm{body}})(g, \mu, x)((g \cdot \xi, \rho, v_x),
  (g\cdot \eta, \sigma,w_x)) = \\
  \omega_{\mathrm{body}}(g, \mu, x)((g \cdot \xi, \rho, v_x),
  (g\cdot \eta, \sigma,w_x)) - \Sigma(\xi, \eta) \\
  - \mathbf{d} \phi_\eta (0, v_x) + \mathbf{d} \phi_\xi(0, w_x)  +
  B(e, x)(\xi_{H \times P}(e, x), \eta_{H \times P}(e, x) ).
\end{multline*}
The last three terms are just the interaction terms.  Hence, by
substituting this expression into the expression for
$\mathcal{S} _\ast\Omega_{\mathrm{body}}$, these terms cancel out, leaving
(\ref{pfomega}).
\end{proof}

In other words, by applying $\mathcal{S} $ we effectively get rid of the
interaction terms in the symplectic form.

\subsection{The Reduced Poisson Structures}

\paragraph{The Poisson Structure on $H \times \mathfrak{h} ^{\ast}  \times P$.}
Now that we have established the different symplectic structures on $H \times
\mathfrak{h} ^{\ast}  \times P$, we can also find an explicit form for the
associated Poisson structures.  

\begin{proposition} \label{prop:poisson}
The Poisson structure on $H \times \mathfrak{h} ^{\ast}  \times P$ associated to
the symplectic structure (\ref{bigsympstruct}) is given by
\begin{align} 
  \{F, K\}_{\mathrm{int}}^0 & = - {\frac{\delta F}{\delta \mu}} \star {\frac{\delta K}{\delta \mu}} - \{F_{|P}, K_{|P}\}_P +
  B\left(\left({\frac{\delta F}{\delta \mu}}\right)_{H \times P}, \left({\frac{\delta K}{\delta \mu}}\right)_{H \times
      P}\right) \nonumber \\
  & - {\frac{\delta F}{\delta \mu}} \lhd K + {\frac{\delta K}{\delta \mu}} \lhd F \label{eq:poisson} \\
  & + \left<\mathbf{d}_g F, g\cdot {\frac{\delta K}{\delta \mu}}\right> - \left<\mathbf{d}_g K, g\cdot {\frac{\delta F}{\delta \mu}}\right> \nonumber
\end{align}
for functions $F = F(g, \mu, x)$ and $K = K(g, \mu, x)$ on $H \times
\mathfrak{h} ^{\ast}  \times P$.
\end{proposition}
\begin{proof}
  The Poisson bracket of two functions $F, K$ on $H \times
  \mathfrak{h} ^{\ast}  \times P$ is defined as $\{F, K\}^0_{\mathrm{int}} = X_K(F)$,
  where $X_K$ is the Hamiltonian vector field associated to $K$,
  defined by $\mathbf{i}_{X_K} \Omega_{\mathrm{body}} = \mathbf{d} K$.
  
  Let $(g, \mu, x)$ be an element of $H \times \mathfrak{h} ^{\ast}  \times P$
  and write $X_K(g, \mu, x) \in T_g G \times \mathfrak{h} ^{\ast}  \times T_x P$
  as $X_K(g, \mu, x) = (g \cdot \xi, \rho, v_x)$, where $\xi \in
  \mathfrak{h} $, $\rho \in \mathfrak{h} ^{\ast} $, and $v_x \in T_x P$.  Note also
  that $\mathbf{d} K$ is of the form 
  \[
    \mathbf{d} K = \left(\mathbf{d}_g K, \frac{\delta K}{\delta \mu}, \mathbf{d}_x K\right),  
  \]
  where $\mathbf{d}_x K$ denotes the differential of $K(g, \mu, x)$ keeping
  $g$ and $\mu$ fixed, and similar for $\mathbf{d}_g K$.

  A long but straightforward calculation shows that
  \begin{equation} \label{components}
    \xi = \frac{\delta K}{\delta \mu} \quad \text{and} \quad  
    v_x = - (\mathbf{d}_x K + \mathbf{d}_x \phi_\xi)^\sharp,
  \end{equation}
  where the map $\sharp : T^{\ast}  P \rightarrow TP$ is given by
  $\sharp(\alpha_x) = w_x$ iff $\mathbf{i}_{w_x} \omega_P = \alpha_x$, while 
  \[
    \rho = - \mathbf{i}_{w} B \circ T_e
    \hat{\sigma}_{(e, x)} - T^{\ast}  L_g(\mathbf{d}_g K), \quad \text{where}
    \quad w = \xi_{H \times P}(e, x) + v_x, 
  \]
  and $v_x$ and $\xi$ are given by (\ref{components}).  Here,
  $\hat{\sigma}_{(g, x)}: H \rightarrow H \times P$ is the map defined
  by $\hat{\sigma}_{(g, x)}(h) = (hg, h\cdot x)$, for all $g, h \in H$
  and $x \in P$.  Note that
  \[
    T_e \hat{\sigma}_{(g, x)}(\xi) = \xi_{H \times P}(g, x).
  \]

  The Poisson bracket is then given by 
  \[
     \{F, K\}^0_{\mathrm{int}} = X_K(F) 
      = \left<g\cdot\xi, \mathbf{d}_g F\right>
      + \left<v_x, \mathbf{d}_x F\right>
      + \left<\rho, \frac{\delta F}{\delta \mu}\right>.
  \]
  The first term is equal to 
  \[
  \left<g\cdot\xi, \mathbf{d}_g F\right> = \left<\frac{\delta K}{\delta \mu},
  T^{\ast}  L_g(\mathbf{d}_g F)\right>, 
  \]
  while for the second term, we have 
  \begin{align*}
    \left<v_x, \mathbf{d}_x F\right> & = - \left<(\mathbf{d}_x K + \mathbf{d}_x \phi_\xi)^\sharp,
      \mathbf{d}_x F\right> \\
    & = -\{F_{|P}, K_{|P}\}_P + \frac{\delta K}{\delta \mu} \lhd F, 
  \end{align*}
  using the definition of the induced Poisson structure in terms of
  the original symplectic structure: $\{F_{|P}, K_{|P}\}_P = (\mathbf{d}_x
  K)^\sharp \cdot \mathbf{d}_x F$.  Writing out the last term requires a
  little more work.  Denote
  \[
    \eta := \frac{\delta F}{\delta \mu};
  \]
  the last term then becomes
  \[
    \left<\rho, \frac{\delta F}{\delta \mu}\right> 
      = -\mathbf{i}_w B\left( \eta_{H \times P} \right) - \left<T^{\ast} 
        L_g(\mathbf{d}_g K), \eta\right>.
  \]
  The first term on the right-hand side can be rewritten as 
  \begin{align*}
    - \mathbf{i}_w B\left( \eta_{H \times P} \right) & = \mathbf{d} \phi_\eta(w) \\
    & = \mathbf{d} \phi_\eta(\xi_{H \times P}) + \mathbf{d}\phi_\eta (v_x) \\
    & = B(\eta_{H \times P}, \xi_{H \times P}) - \mathbf{d}_x \phi_\eta \cdot
    (\mathbf{d}_xK)^\sharp
    - \mathbf{d}_x \phi_\eta \cdot (\mathbf{d}_x \phi_\xi)^\sharp \\
    & = B(\eta_{H \times P}, \xi_{H \times P}) - \eta \lhd K - \eta
    \star \xi.
  \end{align*}
  Putting everything together, we obtain (\ref{eq:poisson}).
\end{proof}

\paragraph{The Reduced Poisson Structure with Interaction Terms.}
Having established the explicit form for the interaction Poisson structure $\{\cdot, \cdot\}_{\mathrm{int}}^0$ on $H
\times \mathfrak{h} ^{\ast}  \times P$, we can at last show that the reduced Poisson structure 
 on $\mathfrak{h} ^{\ast}  \times P$ is given by expression 
 (\ref{redpoisson}) in 
 theorem~\ref{thm:bmr}.

\begin{proof}[Proof of theorem~\ref{thm:bmr}]
According to the Poisson reduction theorem (see \cite{OrRa04}), the reduced Poisson structure
$\{\cdot, \cdot\}_{\mathrm{int}}$ on $\mathfrak{h} ^{\ast}  \times P$ is defined by the
prescription
\[
\{f, h\}_{\mathrm{int}} \circ \pi = \{f \circ \pi, h \circ
\pi\}_{\mathrm{int}}^0, 
\]
where $\{\cdot, \cdot\}_{\mathrm{int}}^0$ is the Poisson structure on $H \times
\mathfrak{h} ^{\ast}  \times P$ described in proposition~\ref{prop:poisson} and
$\pi: H \times \mathfrak{h} ^{\ast}  \times P \rightarrow \mathfrak{h} ^{\ast}  \times P$
forgets the first factor: $\pi(g, \mu, x) = (\mu, x)$.

In other words, the reduced Poisson structure is the original Poisson
structure (\ref{eq:poisson}) evaluated on functions $f = f(\mu, x)$
and $k = k(\mu, x)$ which do not depend on $g \in H$.  This yields the expression (\ref{redpoisson}) for 
$\{\cdot, \cdot\}_{\mathrm{int}}$.
\end{proof}

\paragraph{Shifting Away the Interaction Terms.}  The symplectic form
$\Omega_{\mathrm{body}}$ and the Poisson structure $\{\cdot, \cdot\}_{\mathrm{int}}$ are
characterized by the presence of \textbf{\emph{interaction terms}}.  By this,
we mean that the symplectic structure on $H \times \mathfrak{h} ^{\ast}  \times
P$ is not simply the sum of the symplectic form $\omega_{\mathrm{body}}$ on $H
\times \mathfrak{h} ^{\ast} $ and the symplectic form $\omega_P$ on $P$, but
contains terms involving the curvature $B$ and the $\Bg$-potentials
$\phi_\xi$, and similar for $\{\cdot, \cdot\}_{\mathrm{int}}$.
As shown in lemma~\ref{lemma:pf}, one can remove the interaction terms from 
$\Omega_{\mathrm{body}}$ at the expense of introducing a cocycle $\Sigma$ by using the shift map $\hat{\mathcal{S}}$.  The resulting symplectic structure then induces a reduced Poisson structure on $\mathfrak{h}^\ast \times P$, which is just the natural Poisson structure of 
definition~\ref{def:natpoisson}.

\begin{theorem} \label{thm:poissonsimple} The symplectic form
  $\mathcal{S} _\ast \Omega_{\mathrm{body}}$ on $H \times \mathfrak{h} ^{\ast}  \times P$ is
  $H$-invariant.  The Poisson structure on $\mathfrak{h} ^{\ast}  \times P$
  induced by $\mathcal{S} _\ast \Omega_{\mathrm{body}}$ is the natural Poisson structure $\{\cdot, \cdot\}_\Sigma$ of definition~\ref{def:natpoisson}.
\end{theorem}
\begin{proof}
This is an immediate consequence of the fact that the Lie-Poisson
structure $\{\cdot, \cdot\}_{\mathrm{l.p}}$ is induced by $\omega_{\mathrm{body}}$, while
$\{\cdot, \cdot\}_P$ is induced by $\omega_P$.  
\end{proof}

The map $\mathcal{S}$ shifts away the interacting terms
in the symplectic form, and induces a map on the Lie algebra level.  It is easy to see that this map is just the map $\hat{\mathcal{S}}$ defined in (\ref{hatshift}).  Using this observation, the proof of theorem~\ref{thm:shift} now follows immediately.

\begin{proof}[Proof of theorem~\ref{thm:shift}]
  The map $\hat{\mathcal{S} }$ is the quotient space map induced by $\mathcal{S} $,
  which is a symplectic map taking the symplectic form
  (\ref{bigsympstruct}) with interaction terms into the symplectic
  form (\ref{pfomega}) without interaction terms.  Hence, the induced
  map is a Poisson map with respect to the induced Poisson structures,
  \emph{i.e.} (\ref{poissonmap}) holds.
\end{proof}



\end{document}